\documentclass[review]{elsarticle}

\usepackage[latin1]{inputenc} 
\usepackage{listings, tcolorbox}

\usepackage{lineno,hyperref}
\usepackage{amssymb}
\usepackage{color}
\usepackage{amsmath}
\usepackage{tikz}
\usepackage{amsfonts}
\usepackage{mathrsfs}
\usepackage{amsthm}

\usepackage{bbding}
\usepackage{mathrsfs}
\usepackage{amsmath, amsfonts, amssymb, mathrsfs, txfonts}
\usepackage{graphicx, subfigure}
\usepackage{wasysym}
\usepackage{color}
\usepackage{bm}
\usepackage{enumerate}

\usepackage{pifont}
\usepackage{txfonts}
\usepackage{amsmath}
\usepackage{graphicx}
\usepackage{amsfonts}
\usepackage{amssymb}
\usepackage{mathrsfs,psfrag,eepic,epsfig}
\usepackage{epstopdf}

\modulolinenumbers[5]

\journal{Journal of \LaTeX\ Templates}

\makeatletter \@addtoreset{equation}{section}

\newtheorem{thm}{Theorem}[section]
\newtheorem{cor}[thm]{Corollary}
\newtheorem{lem}[thm]{Lemma}
\newtheorem{prop}[thm]{Proposition}
\theoremstyle{definition}

\newtheorem{rem}[thm]{Remark}
\newtheorem{assum}[thm]{Assumption}
\newtheorem{RHP}[thm]{Riemann-Hilbert Problem}

\renewcommand{\baselinestretch}{1.25}
\begin{document}

\begin{frontmatter}

\title{Soliton resolution for a coupled generalized nonlinear Schr\"{o}dinger equations with weighted Sobolev initial data \tnoteref{mytitlenote}}
\tnotetext[mytitlenote]{
Corresponding author.\\
\hspace*{3ex}\emph{E-mail addresses}: sftian@cumt.edu.cn,
shoufu2006@126.com (S. F. Tian) }

\author{Zhi-Qiang Li, Shou-Fu Tian$^{*}$ and Jin-Jie Yang}
\address{
School of Mathematics, China University of Mining and Technology, Xuzhou 221116, People's Republic of China
}

\begin{abstract}
In this work, we employ the $\bar{\partial}$ steepest descent method in order to study
the Cauchy problem of the cgNLS equations with initial conditions in weighted Sobolev space $H^{1,1}(\mathbb{R})=\{f\in L^{2}(\mathbb{R}): f',xf\in L^{2}(\mathbb{R})\}$. The large time asymptotic behavior of the solution $u(x,t)$ and $v(x,t)$ are derived in a fixed space-time cone $S(x_{1},x_{2},v_{1},v_{2})=\{(x,t)\in\mathbb{R}^{2}: x=x_{0}+vt, ~x_{0}\in[x_{1},x_{2}], ~v\in[v_{1},v_{2}]\}$. Based on the resulting asymptotic behavior, we prove the solution resolution  conjecture of the cgNLS equations which contains the soliton term confirmed by $|\mathcal{Z}(\mathcal{I})|$-soliton on discrete spectrum and the $t^{-\frac{1}{2}}$ order term on continuous spectrum with residual error up to $O(t^{-\frac{3}{4}})$.
\end{abstract}


\end{frontmatter}


\section{Introduction}
The classical nonlinear Schr\"{o}dinger(NLS) equation, i.e.,
\begin{align*}
       iu_{t}\pm u_{xx}+2|u|^{2}u=0,
\end{align*}
is a fundamental physical model and has been applied in various fields such as deep water waves\cite{NLS-1}, plasma physics\cite{NLS-2,NLS-3}, nonlinear optical fibers\cite{NLS-4,NLS-5}, etc. Since it plays a significant role in the field of nonlinear science, lots of researchers and scholars do many works on the NLS equation and its extensions\cite{Tian-PAMS}-\cite{Wangds-2019-JDE}. It is worth noting that in nonlinear science such as optical fibers, the NLS equation only describes the propagation of optical solitons in scalar fields of mono-mode fibers. However, because the change of nonlinear phase stems from the cross-phase modulation in multi-mode fibers or birefringent fibers, we need take the interaction of some field components at different frequencies or polarisations into consideration. The coupled nonlinear Schr\"{o}dinger system came into being to describe the characteristics of such soliton. For example, the well-known Manakov system\cite{Manakov-system}, i.e.
\begin{align*}
iu_{t}\pm \frac{1}{2}u_{xx}+(|u|^{2}+|v|^{2})u=0,\\
iv_{t}\pm \frac{1}{2}v_{xx}+(|u|^{2}+|v|^{2})v=0,
\end{align*}
can be used to describe the intense electromagnetic pulse propagation in birefringent fiber.

In this work, considering that there are more effects of nonlinear factors in practise, we derive a novel integrable coupled generalized nonlinear Schr\"{o}dinger(cgNLS) equation that reads
\begin{gather}
\begin{split}
iu_{t}-i\alpha u_{x}+u_{xx}-2u^{2}v+4\beta^{2}u^{3}v^{2}+4i\beta (uv)_{x}u+\gamma u=0,\\
iv_{t}-i\alpha v_{x}-v_{xx}+2uv^{2}-4\beta^{2}u^{2}v^{3}+4i\beta (uv)_{x}v-\gamma v=0,
\end{split}\label{1.1}
\end{gather}
where $u(x,t)$ and $v(x,t)$ are complex function of variables $x$, $t$, and $\alpha$, $\beta$, $\gamma$ are arbitrary constants. Additionally, we will derive the Lax pair, infinite conservation laws and Hamiltonian function associated with the cgNLS equation.
Furthermore, we investigate the soliton resolution of the cgNLS equation with the initial value condition
\begin{align}
  u(x,0)=u_{0}(x),v(x,0)=v_{0}(x)\in H^{1,1}(\mathbb{R}),
\end{align}
where $H^{1,1}(\mathbb{R})$ is  a weighted Sobolev space, i.e.,
\begin{align}\label{Sobolev-space}
 H^{1,1}(\mathbb{R})=\{f\in L^{2}(\mathbb{R}): f',xf\in L^{2}(\mathbb{R})\}.
\end{align}
In particular, some unique cases of \eqref{1.1} can be obtained by fixing the coefficients $\alpha$, $\beta$ and $\gamma$.
\begin{itemize}
  \item When $\alpha=\gamma=0$ and $\beta\neq 0$, Eq.\eqref{1.1} degenerates into the coupled nonlinear Kundu-Eckhaus equation, i.e.,
      \begin{align*}
       iu_{t}+u_{xx}-2u^{2}v+4\beta^{2}u^{3}v^{2}+4i\beta (uv)_{x}u=0,\\
       iv_{t}-v_{xx}+2uv^{2}-4\beta^{2}u^{2}v^{3}+4i\beta (uv)_{x}v=0,
      \end{align*}
      which is a completely integrable system and possesses a Lax representation, the Hamiltonian structure and some other properties\cite{K-E-geng}-\cite{K-E-JPSJ}.
  \item When $\alpha\neq0$, $\gamma\neq0$, $\beta=0$, and $v=-u^{*}$, Eq.\eqref{1.1} degenerates into the modified Landau-Lifshitz(mLL) equation, i.e.,
      \begin{align*}
       iu_{t}-i\alpha u_{x}+u_{xx}+2|u|^{2}u+\gamma u=0,
      \end{align*}
      which can be employed to depict the dynamic behavior of local magnetization in electromagnetics. Also, a research for the mLL equation reveals the accumulation of energy plays an important role in the generation of magnetic rogue waves\cite{mll-1}. Furthermore, the long time asymptotic behavior of the mLL equation with nonzero boundary condition have been studied\cite{mll-peng}.
  \item When $\alpha=\gamma=\beta=0$, and $v=-u^{*}$, Eq.\eqref{1.1} degenerates into the classical focusing nonlinear Schr\"{o}dinger(NLS) equation, i.e.,
      \begin{align*}
       iu_{t}+u_{xx}+2|u|^{2}u=0,
      \end{align*}
      which has played an important part in nonlinear science.
\end{itemize}
It is not hard to check that the cgNLS equation is not self-adjoint, so that the soliton solution will emerge for zero boundary conditions. Therefore, it is significant to consider the question that how to obtain the long time asymptotic behavior of the cgNLS equation in the domain of solution solutions.

In order to obtain the long time behavior of nonlinear evolution equations, researchers have done a series work. In 1974, using the inverse scattering method, the long time behavior of nonlinear wave equation was first carried out by Manakov \cite{Manakov-1974}. In 1976, the long time asymptotic solutions of nonlinear Schr\"{o}dinger(NLS) equation with decaying initial value was obtained by Zakharov and Manakov \cite{Zakharov-1976}. In 1993, a nonlinear steepest descent method was developed to study the modified Korteweg-de Vries(MKdV) equation and obtained the solution rigorously by Defit and Zhou \cite{Deift-1993}. Then, this approach has been applied widespread. However, it has been shown in literature \cite{Deift-1994-1, Deift-1994-2} that if the initial value is smooth and decay fast enough then the error term is $o(\frac{\log t}{t})$. Then, after the hard work of many researchers \cite{Deift-2003}, the error term becomes $O(t^{-(\frac{1}{2}+\iota)})$ for any $0<\iota<\frac{1}{4}$ with the condition that the initial value belongs to the weighted Sobolev space \eqref{Sobolev-space}.

In recent years, McLaughlin and Miller have made contributions to develop the method for the long time asymptotic analysis of Riemann-Hilbert problem(RHP). Combining steepest descent with $\bar{\partial}$-problem, they presented a $\bar{\partial}$ steepest descent method which has been used to analyse the long time asymptotic behavior of series equations \cite{McLaughlin-1, McLaughlin-2}. Since then, lots of significant works have been done by applying this method. For example, with the assumption of finite mass initial data, Dieng and McLaughlin studied the defocusing NLS equation \cite{Dieng-2008}; with finite density initial data, the defocusing NLS equation was investigated by Cuccagna and Jenkins \cite{Cuccagna-2016}; Borghese and Jenkins studied the Cauchy problem for the focusing NLS equation via applying this method \cite{AIHP}. Of course, there are a series of great work about this $\bar{\partial}$ steepest descent method \cite{Jenkins}-\cite{Faneg-3}.
Compared with the nonlinear steepest descent method, it is not necessary to study the delicate estimates involving $L^{p}$ estimates of Cauchy projection operators by applying the $\bar{\partial}$ steepest descent method. Also, it improves the error estimates, i.e., if the initial value belongs to the weighted Sobolev space \eqref{Sobolev-space}, the error becomes $o(t^{-\frac{3}{4}})$ which has been shown in \cite{Dieng-2008}.

In this work, we are going to use the $\bar{\partial}$ steepest descent method to study the long time asymptotic behavior of the cgNLS equation with the initial value $u_{0}(x)$ and $v_{0}(x)$ which belong to the weighted Sobolev space \eqref{Sobolev-space}.

The outline of this work is as follows. In section 2, based on the Lax pair of the cgNLS equation, we introduce two kinds of eigenfunction and scattering matrix. Also, the analytical, symmetries and asymptotic properties are analyzed. Then the Riemann-Hilbert problem is constructed for the cgNLS equation with initial problem. In section 3, we introduce the matrix function $T(z)$ to separate the jump matric which is defined in \eqref{RH-1} near the phrase point $z_{0}=-\frac{1}{4}\left(\frac{x}{t}+\alpha\right)$. In section 4, we make the continuous extension of the jump matrix off the real axis by introducing a matrix function $R^{(2)}(z)$ and get a mixed $\bar{\partial}$-Riemann-Hilbert(RH) problem. In section 5, we decompose the mixed $\bar{\partial}$-RH problem into two parts which are a model RH problem with $\bar{\partial}R^{(2)}=0$ and a pure $\bar{\partial}$-RH problem with $\bar{\partial}R^{(2)}\neq0$, i.e., $M_{RHP}$ and $M^{(3)}$. In section 6, we solve the model RH problem $M_{RHP}$  via an outer model $M^{(out)}(z)$ for the soliton part and an inner model near the phase point $z_{0}$ which can be solved by a parabolic cylinder model problem. Also, the error function $E(z)$ with a small-norm RH problem is obtained. In section 7, the pure $\bar{\partial}$-RH problem is studied. Finally, in section 8, we obtain the soliton resolution and long time asymptotic behavior of the cgNLS equation.

\section{The spectral analysis of cgNLS equation}

The Lax pair of the equation \eqref{1.1} reads
\begin{gather}
\begin{split}
\psi_{x}=U\psi,\quad \psi_{t}=V\psi, \\
\psi=(\psi_{1},\psi_{2})^{T}, \label{2.1}
\end{split}
\end{gather}
where $\psi_{i}, (i=1,2)$ are eigenfunctions, and
\begin{align}
\begin{split}
U&=-iz\sigma_{3}-i\beta uv\sigma_{3}+U_{0},  \\
V&=-2iz^{2}\sigma_{3}+a_{1}\sigma_{3}+2zU_{0}-i(U_{0})_{x}\sigma_{3}-2\beta U_{0}^{3}+\alpha U_{0}
\end{split}
\end{align}
with $U_{0}=\left(\begin{array}{cc}
0 & u\\
v & 0
\end{array}\right)$, $\sigma_{3}=\left(\begin{array}{cc}
1 & 0\\
0 & -1
\end{array}\right)$ and $a_{1}=\beta(u_{x}v-uv_{x})+4i\beta^{2}u^{2}v^{2}-i(1-\alpha\beta)uv$.
We can check that $U, V$ satisfy the zero curvature equation $U_{t}-V_{x}+[U,V]=0$, which is the compatibility condition of \eqref{1.1}.

Considering the initial conditions that $u_{0}(x), v_{0}(x)\in H^{1,1}(\mathbb{R})$, letting $x\rightarrow\pm\infty$, we can get an asymptotic scattering problem and construct the two Jost solutions, i.e., $\psi_{\pm}\sim e^{-i(zx+(2z^{2}+\alpha z-\frac{1}{2}\gamma)t)\sigma_{3}}$.

 Then, we introduce $m_{\pm}(x,t;z)$ which satisfies that
\begin{align}\label{2.2}
m_{\pm}(x,t;z)=\psi_{\pm}(x,t;z)e^{i(zx+(2z^{2}+\alpha z-\frac{1}{2}\gamma)t)\sigma_{3}}.
\end{align}
The equivalent Lax pair of Eq.\eqref{2.1} can be derived as
\begin{align}
\begin{split}
m_{x}&+iz[\sigma_{3},m]=(-i\beta uv\sigma_{3}+U_{0})m\triangleq U_{1}m,\\
m_{t}&+i(2z^{2}+\alpha z-\frac{1}{2}\gamma)[\sigma_{3},m]=(a_{1}\sigma_{3}+2zU_{0}-i(U_{0})_{x}\sigma_{3}-2\beta U_{0}^{3}+\alpha U_{0})m\triangleq V_{1}m,\label{2.3}
\end{split}
\end{align}
which can be written in full derivative form, i.e.,
\begin{align}\label{fullDerivate1}
d(e^{i(zx+(2z^{2}+\alpha z-\frac{1}{2}\gamma)t)\sigma_{3}}m)=e^{i(zx+(2z^{2}+\alpha z-\frac{1}{2}\gamma)t)\sigma_{3}}(U_{1}dx+V_{1}dt)m.
\end{align}

Evaluating a solution of Eq.\eqref{fullDerivate1} with the form
\begin{align}
m=\mathcal{X}+\frac{m_{1}}{z}+\frac{m_{2}}{z^{2}}+o(z^{-3}), ~~z\rightarrow\infty,\label{2.4}
\end{align}
where $\mathcal{X}$, $m_{1}$ and $m_{2}$ are independent of variable $z$. Substituting Eq.\eqref{2.4} into Eq.\eqref{2.3}, and  comparing the coefficients of the same power of $z$, we show that
\begin{align}\label{2.5}
\begin{split}
\mathcal{X}_{x}&=(-i\beta uv\sigma_{3})\mathcal{X},\\
\mathcal{X}_{t}&=(\beta(u_{x}v-uv_{x})+4i\beta^{2}u^{2}v^{2}-i\alpha\beta uv)\sigma_{3}\mathcal{X}.
\end{split}
\end{align}
We note that Eq.\eqref{1.1} possesses the conservation law
\begin{align}
(-i\beta uv)_{t}=(\beta(u_{x}v-uv_{x})+4i\beta^{2}u^{2}v^{2}-i\alpha\beta uv)_{x}.
\end{align}
Therefore, if we define
\begin{align}
\mathcal{X}(x,t)=e^{i\int_{(-\infty,\infty)}^{(x,t)}\Delta\sigma_{3}},
\end{align}
where
\begin{align}
\Delta(x,t)=-\beta uvdx+(-i\beta(u_{x}v-uv_{x})+4\beta^{2}u^{2}v^{2}-\alpha\beta uv)dt,
\end{align}
the two equations in \eqref{2.5} for $\mathcal{X}$ are consistent and are both satisfied.
Now, introducing a new function $\mu$
\begin{align}
m(x,t,z)=e^{i\int_{(-\infty,\infty)}^{(x,t)}\Delta\sigma_{3}}\mu(x,t,z).
\end{align}
It is obvious that
\begin{align}\label{J-1}
\mu=I+O(\frac{1}{z}),\quad z\rightarrow\infty.
\end{align}
Then, from Eq.\eqref{fullDerivate1}, we can derive that
\begin{align}\label{fullDerivate2}
d(e^{i(zx+(2z^{2}+\alpha z-\frac{1}{2}\gamma)t)\sigma_{3}}\mu)
=e^{i(zx+(2z^{2}+\alpha z-\frac{1}{2}\gamma)t)\sigma_{3}}(U_{2}dx+V_{2}dt)\mu.
\end{align}
where
\begin{gather}
      U_{2}=e^{-i\int_{(-\infty,\infty)}^{(x,t)}\Delta\hat{\sigma}_{3}}\left(\begin{array}{cc}
0& u \\
v & 0
\end{array}\right),\notag\\ V_{2}=e^{-i\int_{(-\infty,\infty)}^{(x,t)}\Delta\hat{\sigma}_{3}}\left(\begin{array}{cc}
iuv & 2zu+iu_{x}-2\beta u^{2}v+\alpha u\\
 2zv-iv_{x}-2\beta uv^{2}+\alpha v & -iuv
\end{array}\right).\notag
\end{gather}
Also, the Lax pair \eqref{2.3} can be transformed into
\begin{gather}
\mu_{x}+iz[\sigma_{3},\mu]=U_{2}\mu,\label{2.6}\\
\mu_{t}+i(2z^{2}+\alpha z-\frac{1}{2}\gamma)[\sigma_{3},\mu]=V_{2}\mu.\label{2.7}
\end{gather}
Based on the Eq.\eqref{fullDerivate2}, we can select two special integration paths i.e., $(-\infty,t)\rightarrow(x,t)$ and $(+\infty,t)\rightarrow(x,t)$, and obtain the following Volterra type integrals
\begin{align}
\begin{matrix}
\mu_{\pm}(x,t;z)=\mathbb{I}+\int_{\pm\infty}^{x}e^{-iz(x-y)\hat{\sigma}_{3}}[U_{2}(y,t;z)\mu_{-}(y,t;z)]\, dy,
\end{matrix}
\end{align}
from which we can derive the analytical properties of $\mu_{\pm}$.
\begin{itemize}
  \item It is assumed that $u(x)-u_{0}, v(x)-v_{0}\in H^{1,1}(\mathbb{R})$. Then, $\mu_{-,1}, \mu_{+,2}$ are analytic in $C_{+}$ and $\mu_{-,2}, \mu_{+,1}$ are analytic in $C_{-}$ and they can be recorded as $\mu^{+}_{-,1}, \mu^{+}_{+,2}, \mu^{-}_{-,2}, \mu^{-}_{+,1}$, respectively. The $\mu_{\pm,j} (j=1,2)$ mean the $j$-th column of $\mu_{\pm}$.
\end{itemize}
For $z\in \mathbb{R}$ both $\psi_{+}$ and $\psi_{-}$ are two fundamental matrix solutions of the scattering problem. Therefore,
\begin{gather}
\psi_{+}(x,t;z)=\psi_{-}(x,t;z)S(z),\quad x,t\in\mathbb{R}, ~~~z\in \mathbb{R},\label{2.8}
\end{gather}
where $S(z)$ is a $2\times2$ matrix and independent of the variable $x$ and $t$.
According to the Abel's theorem \cite{Li-1} and the Lax pair \eqref{2.1}, we know that $(\det \psi)_{x}=(\det \psi)_{t}=0$. Furthermore, we have $\det \psi_{\pm}=\det \mu_{\pm}=1$.Then, we can derive that $\psi_{\pm}$ are reversible. Therefore, we obtain that
\begin{align}\label{2.9}
\begin{split}
&s_{11}(z)=Wr\left(\psi_{+,1},\psi_{-,2}\right),\quad
s_{22}(z)=Wr\left(\psi_{-,1},\psi_{+,2}\right),\\
&s_{12}(z)=Wr\left(\psi_{+,2},\psi_{-,2}\right),\quad
s_{21}(z)=Wr\left(\psi_{-,1},\psi_{+,1}\right),
\end{split}
\end{align}
where $\psi_{\pm,j}$ mean the $j$-column of $\psi_{\pm}$, respectively, and $Wr$ infers to the Wronskians determinant.
Then, according to the analytical properties of $\mu_{\pm}$ and the relationship between $\psi_{\pm}$ and $\mu_{\pm}$, we obtain that $s_{11}$ is analytic in $\mathbb{C}^{+}$, and $s_{22}$ is analytic in $\mathbb{C}^{-}$.

It is easy to check that if $\mu(x,t;z)$ is the solution of Eq.\eqref{2.6} then $-\sigma \mu^{*}(x,z^{*})\sigma$ also is the solution of Eq.\eqref{2.6} and it follows
\begin{gather}
\mu_{\pm}(z)=-\sigma \mu^{*}_{\pm}(z^{*})\sigma. \label{2.10}
\end{gather}
Then, it is obvious that
\begin{gather}
S(z)=-\sigma S^{*}(z^{*})\sigma. \label{2.11}
\end{gather}
Based on Eq.\eqref{J-1} and Eq.\eqref{2.8}, we can derive that $S(z)\rightarrow \mathbb{I}$, $z\rightarrow\infty$.

\begin{assum}\label{assum}
In the following analysis, we make the assumption to avoid the many pathologies possible, i.e.,
\begin{itemize}
  \item For $z\in\mathbb{R}$, no spectral singularities exist, i.e, $s_{11}(z)\neq0$;
  \item Suppose that $s_{11}(z)$ has $N$ zero points, denoted as $\mathcal{Z}=\{(z_{j},Imz_{j}>0)^{N}_{j=1}\}$. So that $s_{22}(\lambda)$ has $N$ zero points $\mathcal{Z}^{*}=\{(z^{*}_{j},Imz^{*}_{j}<0)^{N}_{j=1}\}$.
  \item The discrete spectrum is simple, i.e., if $z_{0}$ is the zero of $s_{11}(z)$, then $s'_{11}(z_{0})\neq0$.
\end{itemize}
\end{assum}

Furthermore, based on Eq.\eqref{2.9} and Eq.\eqref{2.11}, there exists norming constants $c_{j}$ such that
\begin{equation}
\mu_{-,1}(z_{j})=c_{j}e^{2it\theta(z_{j})}\mu_{+,2}(z_{k});~
\mu_{-,2}(z^{*}_{j})=-c^{*}_{j}e^{-2it\theta(z^{*}_{j})}\mu_{+,1}(z^{*}_{j}),\notag
\end{equation}
where $\theta(z)=(z\frac{x}{t}+2z^{2}+\alpha z-\frac{1}{2}\gamma)$.

Now, we introduce a sectionally  meromorphic matrices
\begin{align}\label{Matrix}
M(x,t;z)=\left\{\begin{aligned}
&M^{+}(x,t;z)=\left(\frac{\mu_{-,1}(x,t;z)}{s_{11}(z)},\mu_{+,2}(x,t;z)\right), \quad z\in \mathbb{C}^{+},\\
&M^{-}(x,t;z)=\left(\mu_{+,1}(x,t;z),\frac{\mu_{-,2}(x,t;z)}{s_{22}(z)}\right), \quad z\in \mathbb{C}^{-},
\end{aligned}\right.
\end{align}
where $M^{\pm}(x,t;z)=\lim\limits_{\varepsilon\rightarrow0^{+}}M(x,t;z\pm i\varepsilon),~\varepsilon\in\mathbb{R}$, and  reflection coefficients
\begin{gather}
r(z)=\frac{s_{21}(z)}{s_{11}(z)},~~
 \frac{s_{12}(z)}{s_{22}(z)}=-\frac{s^{*}_{21}(z^{*})}{s^{*}_{11}(z^{*})}
 =-r^{*}(z^{*})=-r^{*}(z),~~z\in\mathbb{R}.
\end{gather}

Based on the above analysis, the matrix function $M(x,t;z)$ admits the following matrix RHP.
\begin{RHP}\label{RH-1}
Find an analysis function $M(x,t;z)$ with the following properties:
\begin{itemize}
  \item $M(x,t;z)$ is meromorphic in $C\setminus\mathbb{R}$.
  \item $M^{+}(x,t;z)=M^{-}(x,t;z)G(x,t;z)$,~~~$z\in\mathbb{R}$,
  where \begin{align}\label{J-Matrix}
G(x,t;z)=\left(\begin{array}{cc}
                   1+|r(z)|^{2} & r^{*}(z)e^{-2it\theta(z)} \\
                   r(z)e^{2it\theta(z)} & 1
                 \end{array}\right);
\end{align}
  \item $M(z)=I+O(z^{-1})$ as $z\rightarrow\infty$;
  \item At each $z_{j}\in\mathcal{Z}$ and $z^{*}_{j}\in\mathcal{Z}^{*}$,  $M(z)$ satisfies the residue condition, i.e.,
      \begin{align}\label{2.12}
\mathop{Res}_{z=z_{j}}M=\lim_{z\rightarrow z_{j}}M\left(\begin{array}{cc}
                   0 & 0 \\
                   c_{j}e^{2it\theta} & 0
                 \end{array}\right),
\mathop{Res}_{z=z^{*}_{j}}M=\lim_{z\rightarrow z^{*}_{j}}M\left(\begin{array}{cc}
                   0 & -c^{*}_{j}e^{-2it\theta} \\
                   0 & 0
                 \end{array}\right).
\end{align}
\end{itemize}
\end{RHP}
\begin{rem}
By referring to the Zhou's vanishing lemma, the existence of the solutions of RHP \ref{RH-1} for $(x,t)\in\mathbb{R}^{2}$ is guaranteed. According to a consequence of Liouville's theorem, we know that if a solution exists, it is unique.
\end{rem}

Then, expanding this solution as $z\rightarrow\infty$ and combining with Eq.\eqref{2.6}, we can reconstruct the solution of cgNLS equation as
\begin{align}\label{2.13}
\begin{split}
u(x,t)=2i \lim_{z\rightarrow \infty}(zM)_{12}e^{2i\int_{(-\infty,\infty)}^{(x,t)}\Delta},~~
v(x,t)=2i \lim_{z\rightarrow \infty}(zM)_{21}e^{-2i\int_{(-\infty,\infty)}^{(x,t)}\Delta}.
\end{split}
\end{align}

\section{Conjugation}
In this section, in order to renormalize the Riemann-Hilbert problem\eqref{RH-1}, we introduce a function to establish a transformation $M\mapsto M^{(1)}$.

In jump matrix \eqref{J-Matrix}, the oscillation term is $e^{2it\theta(z)}$, from which we find a phase point
\begin{align*}
z_{0}=-\frac{1}{4}\left(\frac{x}{t}+\alpha\right).
\end{align*}
Consequently, $\theta(z)$ can be written as
\begin{align}\label{theta}
\theta(z)=2z^{2}-4z_{0}z-\frac{1}{2}\gamma=2(z-z_{0})^{2}-2z^{2}_{0}-\frac{1}{2}\gamma.
\end{align}
And we have $Re(2i\theta)=-8(Rez-z_{0})Imz$. Then, we derive the decaying domains of the oscillation term.

\centerline{\begin{tikzpicture}[scale=0.7]
\path [fill=yellow] (0,4)--(0,0) to (4,0) -- (4,4);
\path [fill=yellow] (0,-4)--(0,0) to (-4,0) -- (-4,-4);
\draw[-][thick](-4,0)--(-3,0);
\draw[-][thick](-3,0)--(-2,0);
\draw[-][thick](-2,0)--(-1,0);
\draw[-][thick](-1,0)--(0,0);
\draw[-][thick](0,0)--(1,0);
\draw[-][thick](1,0)--(2,0);
\draw[-][thick](2,0)--(3,0);
\draw[->][thick](3,0)--(4,0)[thick]node[right]{$Rez$};
\draw[<-][thick](0,4)[thick]node[right]{$Imz$}--(0,3);
\draw[-][thick](0,3)--(0,2);
\draw[-][thick](0,2)--(0,1);
\draw[-][thick](0,1)--(0,0);
\draw[-][thick](0,0)--(0,1);
\draw[-][thick](0,1)--(0,2);
\draw[-][thick](0,2)--(0,3);
\draw[-][thick](0,3)--(0,4);
\draw[-][thick](0,0)--(0,-1);
\draw[-][thick](0,-1)--(0,-2);
\draw[-][thick](0,-2)--(0,-3);
\draw[-][thick](0,-3)--(0,-4);
\draw[fill] (2,1.5)node[below]{\small{$|e^{2i\theta(z)}|\rightarrow0$}};
\draw[fill] (-2.5,-1.5)node[below]{\small{$|e^{2i\theta(z)}|\rightarrow0$}};
\draw[fill] (2,-1.5)node[below]{\small{$|e^{-2i\theta(z)}|\rightarrow0$}};
\draw[fill] (-2.5,1.5)node[below]{\small{$|e^{-2i\theta(z)}|\rightarrow0$}};
\end{tikzpicture}}
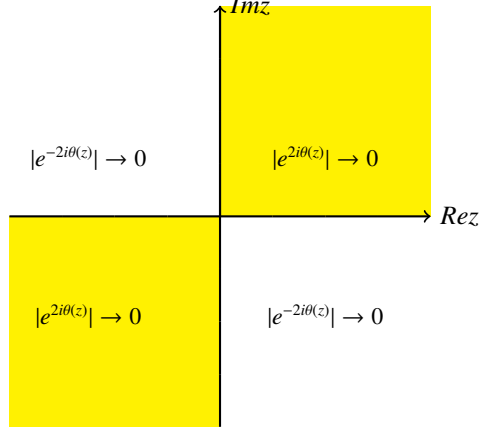
\centerline{\noindent {\small \textbf{Figure 1.} Exponential decaying domains.}}

To make the following analysis more convenient, we introduce some notations.
\begin{align}\label{2.14}
\begin{aligned}
&\triangle^{-}_{z_{0}}=\{k\in\{1,\cdots,N\}|Re(z_{k})<z_{0}\},\\
&\triangle^{+}_{z_{0}}=\{k\in\{1,\cdots,N\}|Re(z_{k})>z_{0}\}.
\end{aligned}
\end{align}
For $\mathcal{I}=[a,b]$, define
\begin{align*}
&\mathcal{Z}(\mathcal{I})=\{z_{k}\in\mathcal{Z}:Rez_{k}\in\mathcal{I}\},\\
&\mathcal{Z}^{-}(\mathcal{I})=\{z_{k}\in\mathcal{Z}:Rez_{k}<a\},\\
&\mathcal{Z}^{+}(\mathcal{I})=\{z_{k}\in\mathcal{Z}:Rez_{k}>b\}.
\end{align*}
For $z_{0}\in\mathcal{I}$, define
\begin{align*}
&\triangle^{-}_{z_{0}}(\mathcal{I})=\{k\in\{1,\cdots,N\}:a\leq Rez_{k}<z_{0}\},\\
&\triangle^{+}_{z_{0}}(\mathcal{I})=\{k\in\{1,\cdots,N\}:a<Rez_{k}\leq z_{0}\}.
\end{align*}

Next, we introduce the function
\begin{align*}
\delta(z)=\exp[i\int_{-\infty}^{z_{0}}\frac{\nu(s)}{s-z}ds],~~\nu(s)=-\frac{1}{2\pi}\log(1+|r(s)|^{2}).
\end{align*}
and
\begin{align}\label{2.15}
T(z)=\prod_{k\in\Delta_{z_{0}}^{-}}\frac{z-z^{*}_{k}}{z-z_{k}}\delta(z),
\end{align}
which has the following properties.
\begin{prop}\label{T-property} $T(z)$ admits that\\
($a$) $T$ is meromorphic in $C\setminus(-\infty, z_{0}]$;\\
($b$) For $z\in C\setminus(-\infty,z_{0}]$, $T^{*}(z^{*})=\frac{1}{T(z)}$;\\
($c$) For $z\in (-\infty,z_{0}]$, the boundary values $T_{\pm}$ satisfy
\begin{align}\label{2.16}
T_{+}(z)/T_{-}(z)=1+|r(z)|^{2}, z\in (-\infty,z_{0}];
\end{align}
($d$) As $|z|\rightarrow \infty $ with $|arg(z)|\leq c<\pi$,
\begin{align}\label{2.17}
T(z)=1+\frac{i}{z}[2\sum_{k\in\Delta_{z_{0}}^{-}}Imz_{k}-\int_{-\infty}^{z_{0}}\nu(s)ds]+O(z^{-2});
\end{align}
($e$) As $z\rightarrow z_{0}$ along any ray $z_{0}+e^{i\phi}R_{+}$ with $|\phi|\leq c<\pi$
\begin{align}\label{2.18}
|T(z-z_{0})-T_{0}(z_{0})(z-z_{0})^{i\nu(z_{0})}|\leq C\parallel r\parallel_{H^{1}(R)}|z-z_{0}|^{\frac{1}{2}},
\end{align}
where $T_{0}(z_{0})$ is the complex unit
\begin{align}\label{2.19}
\begin{split}
&T_{0}(z_{0})=\prod_{k\in\Delta_{z_{0}}^{-}}(\frac{z_{0}-z^{*}_{k}}{z_{0}-z_{k}})e^{i\beta(z_{0},z_{0})},\\
&\beta(z,z_{0})=-\nu(z_{0})\log(z-z_{0}+1)+\int_{-\infty}^{z_{0}}\frac{\nu(s)-\chi(s)\nu(z_{0})}{s-z}ds,
\end{split}
\end{align}
with $\chi(s)=1$ as $s\in(z_{0}-1, z_{0})$, and $\chi(s)=0$ as  $s\in(-\infty, z_{0}-1]$.
\end{prop}

\begin{proof}
For part $(a)$, $(b)$ and $(c)$, it is obvious to obtain their results from the definition of $T(z)$ shown in \eqref{2.15}. For part $(d)$, we expand the $\prod_{k\in\Delta_{z_{0}}^{-}}\frac{z-z^{*}_{k}}{z-z_{k}}$ and $\delta(z)$ for large $z$. Then, a direct calculation can show the result. For part $(e)$, we rewrite \eqref{2.15} as
\begin{align*}
T(z)=T(z,z_{0})&=\prod_{k\in\Delta_{z_{0}}^{-}}\left(\frac{z-z^{*}_{k}}{z-z_{k}}\right) \exp\left(i\int_{z_{0}-1}^{z_{0}}\frac{\nu(z_{0})}{s-z}ds+ i\int_{-\infty}^{z_{0}}\frac{\nu(s)-\chi(s)\nu(z_{0})}{s-z}ds\right)\\
&=\prod_{k\in\Delta_{z_{0}}^{-}}\left(\frac{z-z^{*}_{k}}{z-z_{k}}\right)(z-z_{0})^{i\nu(z_{0})}
\exp(i\beta(z,z_{0}).
\end{align*}
Then, according to the fact that
\begin{align*}
|(z-z_{0})^{i\nu(z_{0})}|\leq|e^{-\pi\nu(z_{0})}|=\sqrt{1+|r(z_{0})|^{2}},
\end{align*}
and the Lemma shown in \cite{Beals-IP}, i.e.,
\begin{align*}
|\beta(z,z_{0})-\beta(z_{0},z_{0})|\leq c||r||_{H^{1}(\mathbb{R})}|z-z_{0}|^{\frac{1}{2}},
\end{align*}
the result of $(e)$ can be obtained directly.
\end{proof}

Now, using the function $T(z)$, we establish a transformation
\begin{align}\label{Trans-1}
M^{(1)}(z)=M(z)T(z)^{-\sigma_{3}},
\end{align}
which admits the following matrix RHP.

\begin{RHP}\label{RH-2}
Find an analysis function $M^{(1)}$ with the following properties:
\begin{itemize}
  \item $M^{(1)}$ is meromorphic on $C\setminus R$;
  \item $M^{(1)}(z)= I+O(z^{-1})$ as $z\rightarrow \infty$;
  \item $M^{(1)}_{\pm}(z)$ satisfy the jump relationship $M^{(1)}_{+}(z)=M^{(1)}_{-}(z)V^{(1)}(z)$, where
      \begin{align}\label{2.20}
       V^{(1)}=\left\{\begin{aligned}
      \left(
        \begin{array}{cc}
      1 & r^{*}(z)T(z)^{2}e^{-2it\theta} \\
      0 & 1 \\
        \end{array}
      \right)\left(
     \begin{array}{cc}
       1 & 0 \\
       r(z)T(z)^{-2}e^{2it\theta} & 1 \\
      \end{array}
    \right)\triangleq W_{L}W_{R},z\in(z_{0}, \infty),\\
   \left(
    \begin{array}{cc}
    1 & 0 \\
    \frac{r(z)T_{-}(z)^{-2}}{1+|r(z)|^{2}}e^{2it\theta} & 1 \\
     \end{array}
   \right)\left(
    \begin{array}{cc}
    1 & \frac{r^{*}(z)T_{+}(z)^{2}}{1+|r(z)|^{2}}e^{-2it\theta} \\
    0 & 1 \\
   \end{array}
  \right)\triangleq U_{L}U_{R},z\in(-\infty, z_{0}).
   \end{aligned}\right.
   \end{align}
   \item $M^{(1)}(z)$ has simple poles at each $z_{k}\in Z$ and $z_{k}\in Z$ at which
\begin{align}\label{2.21}
\begin{split}
\mathop{Res}\limits_{z_{k}}M^{(1)}=\left\{\begin{aligned}
&\lim_{z\rightarrow z_{k}}M^{(1)}\left(\begin{array}{cc}
    0 & c_{k}^{-1}\left((\frac{1}{T})'(z_{k})\right)^{-2}e^{-2it\theta}\\
    0 & 0 \\
  \end{array}
\right),k\in \Delta_{z_{0}}^{-}\\
&\lim_{z\rightarrow z_{k}}M^{(1)}\left(
  \begin{array}{cc}
    0 & 0 \\c_{k}^{-1}T^{-2}(z_{k})e^{2it\theta} & 0 \\
  \end{array}\right),k\in \Delta_{z_{0}}^{+}
\end{aligned}\right.\\
\mathop{Res}\limits_{z^{*}_{k}}M^{(1)}=\left\{\begin{aligned}
&\lim_{z\rightarrow z^{*}_{k}}M^{(1)}\left(\begin{array}{cc}
    0 & 0\\
    -\bar{c}_{k}^{-1}(T'(z^{*}_{k}))^{-2}e^{2it\theta} & 0 \\
  \end{array}
\right),k\in \Delta_{z_{0}}^{-}\\
&\lim_{z\rightarrow z^{*}_{k}}M^{(1)}\left(
  \begin{array}{cc}
    0 & -c^{*}_{k}(T(z^{*}_{k}))^{2}e^{-2it\theta} \\0 & 0 \\
  \end{array}\right),k\in \Delta_{z_{0}}^{+}
\end{aligned}\right.
\end{split}
\end{align}
\end{itemize}
\end{RHP}

\begin{proof}
Firstly, according to the definition of $M^{(1)}$, Proposition \ref{T-property} and the properties of $M(x,t;z)$, we can easily obtain the analyticity, jump matrix and asymptotic behavior of $M^{(1)}$. Then, we consider the residue conditions. When $k\in\triangle^{+}_{z_{0}}$, $T(z)$ is analytic at the points $z_{k}$, $z^{*}_{k}$. Thus, at these points, the residue conditions can be obtained directly from \eqref{2.12} and \eqref{Trans-1}. When $k\in\triangle^{-}_{z_{0}}$, $z_{k}$ is the pole of $T(z)$ and $T^{-1}(z)$ is analytic at $z_{k}$. Consequently, we have
\begin{align*}
\mathop{Res}\limits_{z=z_{k}}M^{(1)}_{1}&=\mathop{Res}\limits_{z=z_{k}}(M_{1}T^{-1})=0,\\
M^{(1)}_{1}(z_{k})&=\lim_{z\rightarrow z_{k}}(M_{1}(z)T^{-1}(z))=c_{k}e^{2it\theta(z_{k})}M_{2}(z_{k})(\frac{1}{T})'(z_{k}),\\
\mathop{Res}\limits_{z=z_{k}}M^{(1)}_{2}&=\mathop{Res}\limits_{z=z_{k}}(M_{2}T)
=M_{2}(z_{k})\lim_{z\rightarrow z_{k}}T(z)(z-z_{k})\\
&=c^{-1}_{k}e^{-2it\theta(z_{k})}M^{(1)}_{1}(z_{k})((\frac{1}{T})'(z_{k}))^{-1},
\end{align*}
from which the first formula shown in \eqref{2.21} can be obtained. Similarly, it is easy to derive the residue condition at $z^{*}_{k}$ with $k\in\triangle^{-}_{z_{0}}$.

\end{proof}

\section{Continuous extension to a mixed $\bar{\partial}$-RH problem}
In this section, we are going to make the continuous extension of the jump matrix off the real axis. It should be pointed that the extension is not necessary and the oscillation term along the new contours are decaying. To achieve this, we define the contours
\begin{align}\label{2.22}
\begin{split}
\Sigma_{j}=&-\frac{1}{4}(\frac{x}{t}+\alpha)+e^{i(2j-1)\pi/4}\mathbb{R}_{+},j=1,2,3,4,\\
\Sigma_{R}=&R\cup\Sigma_{1}\cup\Sigma_{2}\cup\Sigma_{3}\cup\Sigma_{4}.
\end{split}
\end{align}
Moreover, define
\begin{align}
\rho=\frac{1}{2}\min_{\lambda,\zeta\in \mathcal{Z}\cup \mathcal{Z}^{*} \lambda\neq\mu}|\lambda-\zeta|,
\end{align}
and $\chi_{Z} \in C_{0}^{\infty} (C, [0, 1])$ which is supported near the discrete spectrum such that
\begin{align}\label{2.23}
\chi_{Z}(z)=\left\{\begin{aligned}
&1,~~dist(z,\mathcal{Z}\cup \mathcal{Z}^{*})<\rho/3, \\
&0,~~dist(z,\mathcal{Z}\cup \mathcal{Z}^{*})>2\rho/3.
\end{aligned}\right.
\end{align}
It is easy to check that $dist(\mathcal{Z}\cup \mathcal{Z}^{*}, R)>\rho, k=1,2,\cdots,N.$

Next, we need to extend the jump matrix onto the new contours along which oscillation term are decaying. Therefore, we introduce a transformation
\begin{align}\label{Trans-2}
M^{(2)}=M^{(1)}R^{(2)},
\end{align}
where $R^{(2)}$ is selected to admit some conditions.
\begin{itemize}
  \item The purpose of the transformation is to deform the contour $\mathbb{R}$ to the contour $\Sigma^{(2)}$ where $\Sigma^{(2)}=\Sigma_{1}\cup\Sigma_{2}\cup\Sigma_{3}\cup\Sigma_{4}$. So on the real axis, $M^{(2)}$ must have no jump.
  \item To guarantee that the $\bar{\partial}$-contribution to the large-time asymptotic solution of $u(x,t)$ and $v(x,t)$, the norm of $R^{(2)}$ should be controlled.
  \item The introduced transformation need to have no impact on the residue condition.
\end{itemize}
Thus, we define $R^{(2)}$ as
\begin{align}
R^{(2)}=\left\{\begin{aligned}
&\left(
  \begin{array}{cc}
    1 & 0 \\
    R_{1}e^{2it\theta} & 1 \\
  \end{array}
\right)^{-1}\triangleq W_{R}^{-1}, ~&z\in\Omega_{1},\\
&\left(
  \begin{array}{cc}
    1 & R_{3}e^{-2it\theta} \\
    0 & 1 \\
  \end{array}
\right)^{-1}\triangleq U_{R}^{-1}, ~&z\in\Omega_{3},\\
&\left(
  \begin{array}{cc}
    1 & 0 \\
    R_{4}e^{2it\theta} & 1 \\
  \end{array}
\right)\triangleq U_{L}, ~&z\in\Omega_{4},\\
&\left(
  \begin{array}{cc}
    1 & R_{6}e^{-2it\theta} \\
    0 & 1 \\
  \end{array}
\right)\triangleq W_{L},~ &z\in\Omega_{6},\\
&\left(
  \begin{array}{cc}
    1 & 0 \\
    0 & 1 \\
  \end{array}
\right),~ &z\in\Omega_{2}\cup\Omega_{5},
\end{aligned}
\right.
\end{align}
where $R_{j}, j= 1, 3, 4, 6$ are defined in the following proposition.
\begin{prop}\label{R-property}
There exists functions $R_{j}: \Omega_{j} \rightarrow C, j= 1, 3, 4, 6$ such that
\begin{align*}
&R_{1}(z)=\left\{\begin{aligned}&r(z)T^{-2}(z), ~~~~z\in(z_{0}, \infty),\\
&f_{1}=r(z_{0})T_{0}^{-2}(z_{0})(z-z_{0})^{-2i\nu(z_{0})}(1-\chi_{Z}(z)), z\in\Sigma_{1},
\end{aligned}\right.\\
&R_{3}(z)=\left\{\begin{aligned}&\frac{r^{*}(z)}{1+|r(z)|^{2}}T_{+}^{2}(z), ~~~~z\in(-\infty, z_{0}),\\
&f_{3}=\frac{r^{*}(z_{0})}{1+|r(z_{0})|^{2}}T_{0}^{2}(z_{0}) (z-z_{0})^{2i\nu(z_{0})}(1-\chi_{Z}(z)), z\in\Sigma_{2},
\end{aligned}\right.\\
&R_{4}(z)=\left\{\begin{aligned}&\frac{r(z)}{1+|r(z)|^{2}}T_{-}^{2}(z), ~~~~z\in(-\infty, z_{0}),\\
&f_{4}=\frac{r(z_{0})}{1+|r(z_{0})|^{2}}T_{0}^{-2}(z_{0}) (z-z_{0})^{-2i\nu(z_{0})}(1-\chi_{Z}(z)), z\in\Sigma_{3},
\end{aligned}\right.\\
&R_{6}(z)=\left\{\begin{aligned}&r^{*}(z)T^{2}(z), ~~~~z\in(z_{0}, \infty),\\
&f_{6}=r^{*}(z_{0})T_{0}^{2}(z_{0})(z-z_{0})^{2i\nu(z_{0})}(1-\chi_{Z}(z)), z\in\Sigma_{4}.
\end{aligned}\right.
\end{align*}
And $R_{j}$ admit that
\begin{align}\label{R-estimate}
&|R_{j}(z)|\leq c_{1}\sin^{2}(\arg z)+c_{2}\left<Rez\right>^{-1/2},\\
&|\bar{\partial}R_{j}(z)|\leq c_{1}\bar{\partial}\chi_{Z}(z)+c_{2}|z-z_{0}|^{-1/2}+c_{3}|r'(Rez)|,\\
&\bar{\partial}R_{j}(z)=0,z\in \Omega_{2}\cup\Omega_{5},or~ dist(z,\mathcal{Z}\cup\mathcal{Z}^{*})\leq \rho/3.
\end{align}
\end{prop}
where $\left<Rez\right>=\sqrt{1+(Rez)^{2}}$.\\

\centerline{\begin{tikzpicture}[scale=0.7]
\path [fill=yellow] (-5,4)--(-4,4) to (-4,4) -- (0,0);
\path [fill=yellow] (-5,4)--(-4,4) to (-4,4) -- (-4,3);
\path [fill=yellow] (-4,3)--(-5,3) to (-5,3) -- (-5,4);
\path [fill=yellow] (0,0)--(-4,-4) to (-4,-4) -- (-5,-4);
\path [fill=yellow] (-5,-4)--(-5,0) to (-5,0) -- (-5,-4);
\path [fill=yellow] (-5,-3)--(-5,3) to (-5,3) -- (-4,3);
\path [fill=yellow] (-4,3)--(-4,-3) to (-4,-3) -- (-5,-3);
\path [fill=yellow] (-5,-4)--(-4,-4) to (-4,-4) -- (-4,-3);
\path [fill=yellow] (-4,-3)--(-5,-3) to (-5,-3) -- (-5,-4);
\path [fill=yellow] (-5,4)--(-4,4) to (-4,4) -- (0,0);
\path [fill=yellow] (-4,-4)--(0,0) to (0,0) -- (-4,0);
\path [fill=yellow] (-4,0)--(0,0) to (0,0) -- (-4,4);
\path [fill=yellow] (5,4)--(4,4) to (4,4) -- (0,0);
\path [fill=yellow] (0,0)--(4,-4) to (4,-4) -- (5,-4);
\path [fill=yellow] (5,-4)--(5,0) to (5,0) -- (5,4);
\path [fill=yellow] (5,4)--(4,4) to (4,4) -- (4,3);
\path [fill=yellow] (4,3)--(5,3) to (5,3) -- (5,4);
\path [fill=yellow] (4,3)--(5,3) to (5,3) -- (5,-3);
\path [fill=yellow] (5,-3)--(4,-3) to (4,-3) -- (4,3);
\path [fill=yellow] (4,-3)--(5,-3) to (5,-3) -- (5,-4);
\path [fill=yellow] (5,-4)--(4,-4) to (4,-4) -- (4,-3);
\path [fill=yellow] (0,0)--(4,4) to (4,4) -- (4,0);
\path [fill=yellow] (4,0)--(4,-4) to (4,-4) -- (0,0);
\draw[->][thick](-3,0)--(-2,0);
\draw[->][thick](2,0)--(3,0);
\draw[->][thick](5,0)--(6,0)[thick]node[right]{$Rez$};
\draw[->][dashed](-2.5,2.5)--(-2,2);
\draw[->][dashed](1,-1)--(2,-2);
\draw[->][dashed](1,1)--(2,2);
\draw[->][dashed](-2.5,-2.5)--(-2,-2);
\draw [dashed](-4,-4)--(4,4);
\draw [dashed](-4,4)--(4,-4);
\draw[fill] (0,0)node[below]{$z_{0}$};
\draw[fill] (0,-1)node[below]{$\Omega_{5}$};
\draw[fill] (0,1.5)node[below]{$\Omega_{2}$};
\draw[fill] (1.5,-0.5)node[below]{$\Omega_{6}$};
\draw[fill] (1.5,1)node[below]{$\Omega_{1}$};
\draw[fill] (-1.5,1)node[below]{$\Omega_{3}$};
\draw[fill] (-1.5,-0.5)node[below]{$\Omega_{4}$};
\draw[fill] (-6,2)node[below]{$\mathcal {R}^{(2)}=U_{R}^{-1}$};
\draw[fill] (4,2)node[below]{$\mathcal {R}^{(2)}=W_{R}^{-1}$};
\draw[fill] (4,-2)node[below]{$\mathcal {R}^{(2)}=W_{L}$};
\draw[fill] (-6,-2)node[below]{$\mathcal {R}^{(2)}=U_{L}$};
\draw[fill] (4,2.5)node[below]{} circle [radius=0.08];
\draw[fill] (4,-2.5)node[below]{} circle [radius=0.08];
\draw[fill] (2,3)node[below]{} circle [radius=0.08];
\draw[fill] (2,-3)node[below]{} circle [radius=0.08];
\draw[fill] (-2.5,2)node[below]{} circle [radius=0.08];
\draw[fill] (-2.5,-2)node[below]{} circle [radius=0.08];
\draw[fill] (-4,3)node[below]{} circle [radius=0.08];
\draw[fill] (-4,-3)node[below]{} circle [radius=0.08];
\draw[fill] (-2.5,2)node[below]{${z}_{k}$};
\draw[fill] (-2.5,-2)node[below]{${\bar{z}}_{k}$};
\draw[fill] (4.5,4)node[below]{$\Sigma_{1}$};
\draw[fill] (4.5,-4)node[below]{$\Sigma_{4}$};
\draw[fill] (-4.5,4)node[below]{$\Sigma_{2}$};
\draw[fill] (-4.5,-4)node[below]{$\Sigma_{3}$};
\draw[-][thick](-6,0)--(-5,0);
\draw[-][thick](-5,0)--(-4,0);
\draw[-][thick](-4,0)--(-3,0);
\draw[-][thick](-3,0)--(-2,0);
\draw[-][thick](-2,0)--(-1,0);
\draw[-][thick](-1,0)--(0,0);
\draw[-][thick](0,0)--(1,0);
\draw[-][thick](1,0)--(2,0);
\draw[-][thick](2,0)--(3,0);
\draw[-][thick](3,0)--(4,0);
\draw[-][thick](4,0)--(5,0);
\draw[-][thick](5,0)--(6,0);
\end{tikzpicture}}
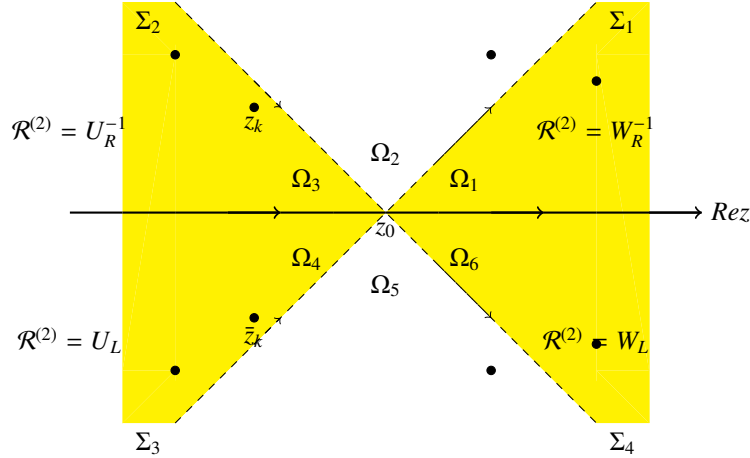
\centerline{\noindent {\small \textbf{Figure 2.}  Definition of $R^{(2)}$ in different domains.}}

Then, using the transformation\eqref{Trans-2} and the definition of $R^{(2)}$, we derive that $M^{(2)}$ admits the following mixed $\bar{\partial}$-RH problem.
\begin{RHP}\label{RH-3}
Find a matrix value function $M^{(2)}$, admitting
\begin{itemize}
 \item $M^{(2)}(x,t,z)$ is continuous in $\mathbb{C}\backslash(\Sigma^{(2)}\cup\mathcal{Z}\cup\mathcal{Z}^{*})$.
 \item $M_+^{(2)}(x,t,z)=M_-^{(2)}(x,t,z)V^{(2)}(x,t,z),$ \quad $z\in\Sigma^{(2)}$, where the jump matrix $V^{(2)}(x,t,z)$ satisfies
 \begin{align}\label{J-V2}
 V^{(2)}(z)=(R_-^{(2)})^{-1}V^{(1)}R_+^{(2)}=I+(1-\chi_{\mathcal{Z}}(z)) \bar{V}^{(2)},
 \end{align}
 with
 \begin{align}
\bar{V}^{(2)}(z)=\left\{\begin{aligned}
&\begin{pmatrix}0&0\\r(z_0)T_0(z_0)^{-2}(z-z_0)^{-2i\nu(z_0)}e^{2it\theta}&0\end{pmatrix} &z\in\Sigma_1, \\
&\begin{pmatrix}0&\frac{r^{*}(z_0)T_0(z_0)^2}{1+|{r(z_0})|^2}(z-z_0)^{2i\nu(z_0)}e^{-2it\theta}\\0&0\end{pmatrix} &z\in\Sigma_2, \\
&\begin{pmatrix}0&0\\\frac{r(z_0)T_0(z_0)^{-2}}{1+|{r(z_0})|^2}(z-z_0)^{-2i\nu(z_0)}e^{2it\theta}&0\end{pmatrix} &z\in\Sigma_3, \\
&\begin{pmatrix}0&r^{*}(z_0)T_0(z_0)^{2}(z-z_0)^{2i\nu(z_0)}e^{-2it\theta}\\0&0\end{pmatrix} &z\in\Sigma_4.
\end{aligned}\right.
\end{align}
\item $M^{(2)}(x,t,z)\rightarrow I,$ \quad $z\rightarrow\infty$.
\item For $\mathbb{C}\backslash(\Sigma^{(2)}\cup\mathcal{Z}\cup\mathcal{Z}^{*})$, $\bar{\partial}M^{(2)}=M^{(2)}\bar{\partial}\mathcal{R}^{(2)}(z),$ where
    \begin{align}\label{4.1}
\bar{\partial}\mathcal{R}^{(2)}=\left\{\begin{aligned}
&\begin{pmatrix}0&0\\-\bar{\partial}R_1 e^{2it\theta}&0\end{pmatrix}, &z\in\Omega_1, \\
&\begin{pmatrix}0&-\bar{\partial}R_3 e^{-2it\theta}\\0&0\end{pmatrix}, &z\in\Omega_3, \\
&\begin{pmatrix}0&0\\ \bar{\partial}R_4 e^{2it\theta}&0\end{pmatrix}, &z\in\Omega_4, \\
&\begin{pmatrix}0&-\bar{\partial}R_6 e^{-2it\theta}\\0&0\end{pmatrix}, &z\in\Omega_6,\\
&\begin{pmatrix}0&0\\0&0\end{pmatrix}, &z\in\Omega_2\cup\Omega_5.
\end{aligned}\right.
\end{align}
  \item  $M^{(2)}$ admits the residue conditions at poles $z_{k} \in \mathcal{Z}$ and $z^{*}_{k} \in \mathcal{Z}^{*}$, i.e.,
      \begin{align}
\mathop{Res}_{z=z_k}M^{(2)}=\left\{\begin{aligned}
&\mathop{lim}_{z\rightarrow z_k}M^{(2)}\begin{pmatrix}0&c_k^{-1}(\frac{1}{T})^{'}(z_k)^{-2}e^{-2it\theta(z_k)}\\0&0\end{pmatrix}, &&k\in \triangle_{z_0}^{-}\\
&\mathop{lim}_{z\rightarrow z_k}M^{(2)}\begin{pmatrix}0&0\\c_k^{-1}T(z_k)^{-2}e^{2it\theta(z_k)}&0\end{pmatrix}, &&k\in \triangle_{z_0}^+
\end{aligned}\right.\notag\\
\mathop{Res}_{z=z^{*}_{k}}M^{(2)}=\left\{\begin{aligned}
&\mathop{lim}_{z\rightarrow z^{*}_{k}}M^{(2)}\begin{pmatrix}0&0\\-(c^{*}_k)^{-1}T^{'}(\bar{z}_k)^{-2}e^{-2it\theta(\bar{z}_k)}&0\end{pmatrix}, &&k\in \triangle_{z_0}^{-}\\
&\mathop{lim}_{z\rightarrow z^{*}_{k}}M^{(2)}\begin{pmatrix}0&-c^{*}_kT(\bar{z}_k)^{2}e^{-2it\theta(\bar{z}_k)}\\0&0\end{pmatrix}, &&k\in \triangle_{z_0}^+
\end{aligned}\right.\notag
\end{align}
\end{itemize}
\end{RHP}
Then, from Eq.\eqref{2.13} and the transformation \eqref{Trans-1} and \eqref{Trans-2}, we can reconstruct the solution of cgNLS equation as
\begin{align}\label{solution2}
\begin{split}
u(x,t)=2i \lim_{z\rightarrow \infty}(zM^{(2)})_{12}e^{2i\int_{(-\infty,\infty)}^{(x,t)}\Delta},~~
v(x,t)=2i \lim_{z\rightarrow \infty}(zM^{(2)})_{21}e^{-2i\int_{(-\infty,\infty)}^{(x,t)}\Delta}.
\end{split}
\end{align}

\centerline{\begin{tikzpicture}[scale=0.7]
\path [fill=yellow] (-5,4)--(-4,4) to (-4,4) -- (0,0);
\path [fill=yellow] (-5,4)--(-4,4) to (-4,4) -- (-4,3);
\path [fill=yellow] (-4,3)--(-5,3) to (-5,3) -- (-5,4);
\path [fill=yellow] (0,0)--(-4,-4) to (-4,-4) -- (-5,-4);
\path [fill=yellow] (-5,-4)--(-5,0) to (-5,0) -- (-5,-4);
\path [fill=yellow] (-5,-3)--(-5,3) to (-5,3) -- (-4,3);
\path [fill=yellow] (-4,3)--(-4,-3) to (-4,-3) -- (-5,-3);
\path [fill=yellow] (-5,-4)--(-4,-4) to (-4,-4) -- (-4,-3);
\path [fill=yellow] (-4,-3)--(-5,-3) to (-5,-3) -- (-5,-4);
\path [fill=yellow] (-5,4)--(-4,4) to (-4,4) -- (0,0);
\path [fill=yellow] (-4,-4)--(0,0) to (0,0) -- (-4,0);
\path [fill=yellow] (-4,0)--(0,0) to (0,0) -- (-4,4);
\path [fill=yellow] (5,4)--(4,4) to (4,4) -- (0,0);
\path [fill=yellow] (0,0)--(4,-4) to (4,-4) -- (5,-4);
\path [fill=yellow] (5,-4)--(5,0) to (5,0) -- (5,4);
\path [fill=yellow] (5,4)--(4,4) to (4,4) -- (4,3);
\path [fill=yellow] (4,3)--(5,3) to (5,3) -- (5,4);
\path [fill=yellow] (4,3)--(5,3) to (5,3) -- (5,-3);
\path [fill=yellow] (5,-3)--(4,-3) to (4,-3) -- (4,3);
\path [fill=yellow] (4,-3)--(5,-3) to (5,-3) -- (5,-4);
\path [fill=yellow] (5,-4)--(4,-4) to (4,-4) -- (4,-3);
\path [fill=yellow] (0,0)--(4,4) to (4,4) -- (4,0);
\path [fill=yellow] (4,0)--(4,-4) to (4,-4) -- (0,0);
\draw(4,2.5) [black, line width=0.5] circle(0.3);
\draw(4,-2.5) [black, line width=0.5] circle(0.3);
\draw(2,3) [black, line width=0.5] circle(0.3);
\draw(2,-3) [black, line width=0.5] circle(0.3);
\draw(-2.5,2) [black, line width=0.5] circle(0.3);
\draw(-2.5,-2) [black, line width=0.5] circle(0.3);
\draw(-4,3) [black, line width=0.5] circle(0.3);
\draw(-4,3) [black, line width=0.5] circle(0.3);
\draw[fill][white] (4,2.5) circle [radius=0.3];
\draw[fill][white] (4,-2.5) circle [radius=0.3];
\draw[fill][white] (-2.5,2) circle [radius=0.3];
\draw[fill][white] (-2.5,-2) circle [radius=0.3];
\draw[fill][white] (-4,3) circle [radius=0.3];
\draw[fill][white] (-4,-3) circle [radius=0.3];
\draw[->][thick](-3,0)--(-2,0);
\draw[->][thick](2,0)--(3,0);
\draw[->][thick](5,0)--(6,0)[thick]node[right]{$Rez$};
\draw[->][dashed](-2.5,2.5)--(-2,2);
\draw[->][dashed](1,-1)--(2,-2);
\draw[->][dashed](1,1)--(2,2);
\draw[->][dashed](-2.5,-2.5)--(-2,-2);
\draw [dashed](-4,-4)--(4,4);
\draw [dashed](-4,4)--(4,-4);
\draw[fill] (0,0)node[below]{$z_{0}$};
\draw[fill] (0,-1)node[below]{$\Omega_{5}$};
\draw[fill] (0,1.5)node[below]{$\Omega_{2}$};
\draw[fill] (1.5,-0.5)node[below]{$\Omega_{6}$};
\draw[fill] (1.5,1)node[below]{$\Omega_{1}$};
\draw[fill] (-1.5,1)node[below]{$\Omega_{3}$};
\draw[fill] (-1.5,-0.5)node[below]{$\Omega_{4}$};
\draw[fill] (-6,2)node[below]{$\mathcal {R}^{(2)}=U_{R}^{-1}$};
\draw[fill] (6,2)node[below]{$\mathcal {R}^{(2)}=W_{R}^{-1}$};
\draw[fill] (6,-2)node[below]{$\mathcal {R}^{(2)}=W_{L}$};
\draw[fill] (-6,-2)node[below]{$\mathcal {R}^{(2)}=U_{L}$};
\draw[fill] (4,2.5)node[below]{} circle [radius=0.08];
\draw[fill] (4,-2.5)node[below]{} circle [radius=0.08];
\draw[fill] (2,3)node[below]{} circle [radius=0.08];
\draw[fill] (2,-3)node[below]{} circle [radius=0.08];
\draw[fill] (-2.5,2)node[below]{} circle [radius=0.08];
\draw[fill] (-2.5,-2)node[below]{} circle [radius=0.08];
\draw[fill] (-4,3)node[below]{} circle [radius=0.08];
\draw[fill] (-4,-3)node[below]{} circle [radius=0.08];
\draw[fill] (-2.5,2)node[below]{${z}_{k}$};
\draw[fill] (-2.5,-2)node[below]{${\bar{z}}_{k}$};
\draw[fill] (4.5,4)node[below]{$\Sigma_{1}$};
\draw[fill] (4.5,-4)node[below]{$\Sigma_{4}$};
\draw[fill] (-4.5,4)node[below]{$\Sigma_{2}$};
\draw[fill] (-4.5,-4)node[below]{$\Sigma_{3}$};
\draw[-][thick](-6,0)--(-5,0);
\draw[-][thick](-5,0)--(-4,0);
\draw[-][thick](-4,0)--(-3,0);
\draw[-][thick](-3,0)--(-2,0);
\draw[-][thick](-2,0)--(-1,0);
\draw[-][thick](-1,0)--(0,0);
\draw[-][thick](0,0)--(1,0);
\draw[-][thick](1,0)--(2,0);
\draw[-][thick](2,0)--(3,0);
\draw[-][thick](3,0)--(4,0);
\draw[-][thick](4,0)--(5,0);
\draw[-][thick](5,0)--(6,0);
\end{tikzpicture}}
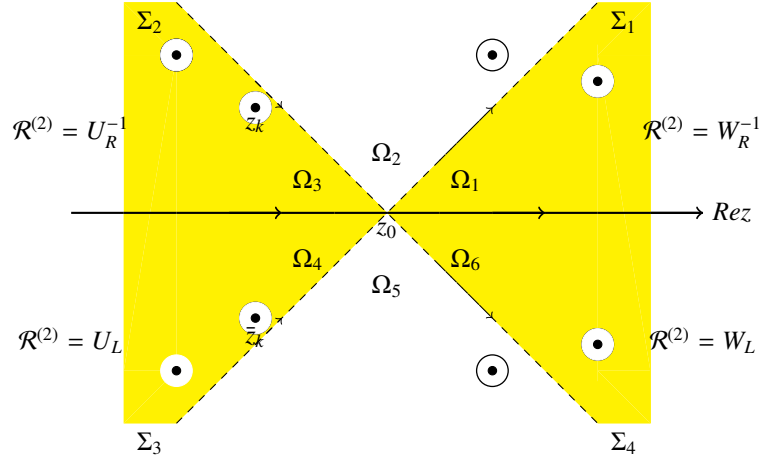
\noindent {\small \textbf{Figure 3.}   Jump matrix $V^{(2)}$, yellow parts support $\bar{\partial}$ derivative: $\bar{\partial}R^{(2)}\neq0$. White parts do not support $\bar{\partial}$ derivative: $\bar{\partial}R^{(2)}=0$.}

\section{Decomposition of the mixed $\bar{\partial}$-RH problem}
In this section, we are going to decompose the mixed $\bar{\partial}$-RH problem, i.e., RHP \ref{RH-3} into two parts which include a model RH problem with $\bar{\partial}R^{(2)}=0$ and a pure $\bar{\partial}$-RH problem. We denote $M_{RHP}$ as the solution when $\bar{\partial}R^{(2)}=0$ in the mixed $\bar{\partial}$-RH problem \ref{RH-3}. Then, if the existence of the solution of $M_{RHP}$ can be proved and its asymptotic expansion for large $t$ can be constructed, the  mixed $\bar{\partial}$-RH problem \ref{RH-3} can be reduced to a pure $\bar{\partial}$-RH problem. Now, we first show the pure $\bar{\partial}$-RH problem. The proof of the existence and asymptotic of $M_{RHP}$ will be shown in the following analysis.

\begin{RHP}\label{RH-rhp}
Find a matrix value function $M_{RHP}$, admitting
\begin{itemize}
 \item $M_{RHP}$ is analytical in $\mathbb{C}\backslash(\Sigma^{(2)}\cup\mathcal{Z}\cup\mathcal{Z}^{*})$;
 \item $M_{RHP,+}(x,t,z)=M_{RHP,-}(x,t,z)V^{(2)}(x,t,z),$ \quad $z\in\Sigma^{(2)}$, where $V^{(2)}(x,t,z)$ is the same with the jump matrix appears in RHP \ref{RH-3};
 \item As $z\rightarrow\infty$, $M_{RHP}(x,t,z)=I+o(z^{-1})$;
 \item $M_{RHP}$ possesses the same residue condition with $M^{(2)}$.
 \end{itemize}
\end{RHP}

Now, using the $M_{RHP}$, we construct a transformation
\begin{align}\label{delate-pure-RHP}
M^{(3)}(z)=M^{(2)}(z)M_{RHP}(z)^{-1}.
\end{align}
Then, we can get the following pure $\bar{\partial}$-RH problem.
\begin{RHP}\label{RH-4}
Find a matrix value function $M^{(3)}$, admitting
\begin{itemize}
 \item $M^{(3)}$ is continuous with sectionally continuous first partial derivatives in $\mathbb{C}\backslash(\Sigma^{(2)}\cup\mathcal{Z}\cup\mathcal{Z}^{*})$;
 \item For $z\in \mathbb{C}$, we obtain $\bar{\partial}M^{(3)}(z)=M^{(3)}(z)W^{(3)}(z)$,
       where
       \begin{align}\label{5.1}
       W^{(3)}=M_{RHP}^{(2)}(z)\bar{\partial}R^{(2)}M_{RHP}^{(2)}(z)^{-1};
       \end{align}
 \item As $z\rightarrow\infty$,
       \begin{align}
       M^{(3)}(z)=I+o(z^{-1}).
       \end{align}
 \end{itemize}
\end{RHP}
\begin{proof}
Based on the properties of the $M_{RHP}$ and $M^{(2)}$ that has been shown in RHP \ref{RH-rhp} and \eqref{RH-3}, we can derive the analyticity and asymptotic properties of $M^{(3)}$ easily. According to the construction of the $M_{RHP}$, we know that $M_{RHP}$ possesses the same jump matrix with $M^{(2)}$. Consequently, we obtain that
\begin{align*}
M^{(3)}_{-}(z)^{-1}M^{(3)}_{+}(z)&=M_{RHP,-}(z)M^{(2)}_{-}(z)^{-1}M^{(2)}_{+}(z)M_{RHP,+}(z)^{-1}\\
&=M_{RHP,-}(z)V^{2}(z)(M_{RHP,-}(z)V^{2}(z))^{-1}=\emph{I},
\end{align*}
from which we know that $M^{(3)}$ has no jump. Next, we futher explain that there exists no pole in $M^{(3)}$. We use $N_{k}$ to denote nilpotent matrix which appears in the left side of the residue condition of RHP \ref{RH-3} and RHP \ref{RH-rhp}. Then, we obtain the Laurent expansions
\begin{align*}
M^{(2)}(z)=C(z_{k})\left[\frac{N_{k}}{z-z_{k}}+\emph{I}\right]+o(z-z_{k}),\\
M_{RHP}(z)=\hat{C}(z_{k})\left[\frac{N_{k}}{z-z_{k}}+\emph{I}\right]+o(z-z_{k}),
\end{align*}
where $C(z_{k})$ and $\hat{C}(z_{k})$ are constant terms. Then, we can derive that
\begin{align*}
M^{(2)}(z)M_{RHP}(z)^{-1}=o(1),
\end{align*}
which means has removable singularities at $z_{k}$. Finally, based on the definition of $M^{(3)}$, we have
\begin{align*}
\bar{\partial}M^{(3)}(z)&=\bar{\partial}(M^{(2)}(z)M_{RHP}(z)^{-1})
=\bar{\partial}M^{(2)}(z)M_{RHP}(z)^{-1}
=M^{(2)}(z)\bar{\partial}R^{(2)}(z)M_{RHP}(z)^{-1}\\
&=M^{(2)}(z)M_{RHP}(z)^{-1}(M_{RHP}(z)\bar{\partial}R^{(2)}(z)M_{RHP}(z)^{-1})=M^{(3)}(z)W^{(3)}(z),
\end{align*}
i.e., the second condition of RHP \ref{RH-4}.
\end{proof}

\section{The pure RH problem}
In this section, we construct the solution $M_{RHP}$ of RHP \ref{RH-rhp}. Define
\begin{align*}
\mathcal{U}_{z_0}=\{z:|z-z_0|<\rho/2\},
\end{align*}
and we decompose $M_{RHP}$ into two parts
\begin{align}\label{Mrhp}
M_{RHP}(z)=\left\{\begin{aligned}
&E(z)M^{(out)}(z), &&z\in\mathbb{C}\backslash \mathcal{U}_{z_0},\\
&E(z)M^{(out)}(z)M^{(pc)}(z_0,r_0), &&z\in \mathcal{U}_{z_0},
\end{aligned} \right.
\end{align}
where $M^{(out)}$ solve a model RHP, $M^{(pc)}$ is a known parabolic cylinder model and $E(z)$ is an error function which is a solution of a small-norm Riemann-Hilbert problem.
\subsection{Outer model RH problem: $M^{(out)}$}
Considering the construction of $M_{RHP}$, we know that $M_{RHP}$ admits a pure Riemann-Hilbert problem. Additionally, for the jump matrix $V^{(2)}$, we have the following estimate.
\begin{align}\label{V2-I}
||V^{(2)}-\emph{I}||_{L^{\infty}(\Sigma^{(2)})}=\left\{\begin{aligned}
&o(e^{-4t|z-z_0|^2}), &&z\in\Sigma^{(2)}\backslash \mathcal{U}_{z_0},\\
&o|z-z_0|^{-1}t^{-\frac{1}{2}}, &&z\in\Sigma^{(2)}\cap\mathcal{U}_{z_0}.
\end{aligned}\right.
\end{align}
Recall that $z_{0}=-\frac{1}{4}\left(\frac{x}{t}+\alpha\right)$.
Due to $|z-z_{0}|\geqslant\rho/2$ outside $\mathcal{U}_{z_0}$, if we omit the jump condition of $M_{RHP}$, the estimate infers to that exponentially small error exists. Next, we establish a model RH problem and prove that its solution can be approximated by finite sum of soliton solutions.
\begin{RHP}\label{RH-5}
Find a matrix value function $M^{(out)}$, admitting
\begin{itemize}
  \item $M^{(out)}(x,t;z)$ is analytical in $\mathbb{C}\backslash(\Sigma^{(2)}\cup\mathcal{Z}\cup\mathcal{Z}^{*})$;
  \item As $z\rightarrow\infty$,
       \begin{align}
       M^{(out)}(x,t;z)=I+o(z^{-1});
       \end{align}
  \item $M^{(out)}(x,t;z)$ has simple poles at each point in $\mathcal{Z}\cup\mathcal{Z}^{*}$ admitting the same residue condition in RHP \ref{RH-3} with $M^{(out)}(x,t;z)$ replacing $M^{(2)}(x,t;z)$.
\end{itemize}
\end{RHP}

To obtain the solution of $M^{(out)}(x,t;z)$, we first study RHP \ref{RH-1} for the case of reflectionless. Under this condition, we know that $M$ has no jump and get the following Riemann-Hilbert problem from RHP \ref{RH-1}.
\begin{RHP}\label{RH-6}
Find a matrix value function $M(x,t;z|\sigma_{d})$, admitting
\begin{itemize}
  \item $M(x,t;z|\sigma_{d})$ is analytical in $\mathbb{C}\setminus(\mathcal{Z}\bigcup\mathcal{Z}^{*})$;
  \item $M(x,t;z|\sigma_{d})=I+O(z^{-1})$, \quad $z\rightarrow\infty$;
  \item $M(x,t;z|\sigma_{d})$ satisfies the following residue conditions at simple poles $z_{k}\in\mathcal{Z}$ and $z_{k}^{*}\in\mathcal{Z}^{*}$
\begin{align}
\begin{aligned}
&\mathop{Res}_{z=z_{k}}M(x,t;z|\sigma_{d})=\mathop{lim}_{z\rightarrow z_{k}}M(x,t;z|\sigma_{d})N_{k},\\
&\mathop{Res}_{z=z_{k}^{*}}M(x,t;z|\sigma_{d})=\mathop{lim}_{z\rightarrow z_{k}^{*}}M(x,t;z|\sigma_{d})\sigma_{2}N^{*}_{k}\sigma_{2},
\end{aligned}
\end{align}
where $\sigma_{d}=\{(z_{k}, c_{k}), z_{k}\in\mathcal{Z}\}^{N}_{k=1}$, which satisfies $z_{k}\neq z_{j}$ for $k\neq j$, are scattering data , and
\begin{gather}
N_{k}=\left(\begin{aligned}
\begin{array}{cc}
  0 & 0 \\
  \gamma_{k}(x,t) & 0
\end{array}
\end{aligned}\right),~
\gamma_{k}(x,t)=c_{k}e^{2it\theta(z_{k})},\\
\theta(z_{k})=2z^{2}+\alpha z-\gamma+\frac{x}{t}z.
\end{gather}
\end{itemize}
\end{RHP}
Then, according to the Liouville's theorem, the uniqueness of the solution is obvious.  Based on the symmetry that shown in Eq.\eqref{2.10}, we obtain $M(x,t;z|\sigma_{d})=-\sigma M^{*}(x,t;z^{*}|\sigma_{d})\sigma$, from which we can derive the following expansion, i.e.,
\begin{align}
M(x,t;z|\sigma_{d})=\emph{I}+\sum_{k=1}^{N}\left[\frac{1}{z-z_{k}}\left(\begin{aligned}
\begin{array}{cc}
  \zeta_{k}(x,t) & 0 \\
  \eta_{k}(x,t) & 0
\end{array}
\end{aligned}\right)+\frac{1}{z-z^{*}_{k}}\left(\begin{aligned}
\begin{array}{cc}
  0 & -\eta^{*}_{k}(x,t) \\
  0 & \zeta^{*}_{k}(x,t)
\end{array}
\end{aligned}\right)\right],
\end{align}
where $\zeta_{k}(x,t)$ and $\eta_{k}(x,t)$ are unknown coefficients to be determined. Next, using the similar way that has been shown in literature\cite{AIHP}, we obtain the following proposition.
\begin{cor}
For the given scattering data $\sigma_{d}$, the RHP \ref{RH-6} exists the unique solution.
\end{cor}
So, for given scattering data $\sigma_{d}$, the unique solution of RHP\eqref{RH-6} can be shown that
\begin{align}
u_{sol}(x,t;\sigma_{d})=2i\mathop{lim}_{z\rightarrow \infty}(zM(z|\sigma_{d}))_{12}e^{2i\int_{(-\infty,\infty)}^{(x,t)}\Delta},\\
v_{sol}(x,t;\sigma_{d})=2i\mathop{lim}_{z\rightarrow \infty}(zM(z|\sigma_{d}))_{21}e^{-2i\int_{(-\infty,\infty)}^{(x,t)}\Delta}.
\end{align}

Now, we need to establish the relation between $M(x,t;z|\sigma_{d})$ and $M^{(out)}(x,t;z)$. We first give the following proposition, and further explained it in the following two subsections.
\begin{cor}\label{prop-6.2}
There exist unique solution $M^{(out)}$ of RHP \ref{RH-5}. Particularly,
\begin{gather}
M^{(out)}(z)=M^{\vartriangle_{z_{0}}^{-}}(z)\delta(z)^{\sigma_{3}}
=M^{\vartriangle_{z_{0}}^{-}}(z|\sigma_{d}^{out})
\end{gather}
where $M^{\vartriangle_{z_{0}}^{-}}(z)$ is the solution of RHP \ref{RH-7} with $\vartriangle=\vartriangle_{z_{0}}^{-}$ and $\sigma_{d}^{out}=\{(z_{k},\widetilde{c}_{k}(z_{0}))\}_{k=1}^{N}$ with
\begin{align}
\widetilde{c}_{k}(z_{0})= c_{k}e^{\frac{i}{\pi}\int_{-\infty}^{z_{0}}\frac{log(1+|r(s)|^{2})}{s-z_{k}}ds}.
\end{align}
In addition,
\begin{align}
u_{sol}(x,t;\sigma_{d}^{out})=\lim_{z\rightarrow\infty}2izM_{12}^{(out)}(x,t;z)
e^{2i\int_{(-\infty,\infty)}^{(x,t)}\Delta},\\
v_{sol}(x,t;\sigma_{d}^{out})=\lim_{z\rightarrow\infty}2izM_{21}^{(out)}(x,t;z)
e^{-2i\int_{(-\infty,\infty)}^{(x,t)}\Delta},
\end{align}
where $u_{sol}(x,t;\sigma_{d}^{out})$ and $v_{sol}(x,t;\sigma_{d}^{out})$ are the $N$-soliton solution of Eq.\eqref{1.1} corresponding the scattering data $\sigma_{d}^{out}$.
\end{cor}

\subsubsection{Renormalization of the Riemann-Hilbert problem for reflectionless case}
For the case of reflectionless, recall that
\begin{align}
M(x,t;z|\sigma_{d})=\left(\frac{\mu_{-,1}(x,t;z)}{s_{11}(z)},\mu_{+,2}(x,t;z)\right),~~ z\in \mathbb{C}^{+},~~s_{11}(z)=\prod_{k=1}^{N}\left(\frac{z-z_{k}}{z-z^{*}_{k}}\right).
\end{align}
Let $\vartriangle\subseteq\{1,2,\cdots,N\}$, $\bigtriangledown\subseteq\{1,2,\cdots,N\}\setminus\Delta$, and define
\begin{align}
s_{11,\vartriangle}=\prod_{k\in\vartriangle}\frac{z-z_{k}}{z-z^{*}_{k}},\quad
s_{11,\triangledown}=\frac{s_{11}}{s_{11,\vartriangle}}=
\prod_{k\in\triangledown}\frac{z-z_{k}}{z-z^{*}_{k}}.
\end{align}
Then, we introduce the normalization transformation
\begin{align}
M^{\vartriangle}(x,t;z|\sigma_{d}^{\vartriangle})=M(x,t;z|\sigma_{d})s_{11,\vartriangle}(z)^{\sigma_{3}},
\end{align}
which splits the poles between the columns of $M(x,t;z|\sigma_{d})$ by selecting different $\vartriangle$. Then, we can get the modified RIemann-Hilbert problem.
\begin{RHP}\label{RH-7}
Given scattering data $\sigma_{d}=\{(z_{k}, c_{k})\}^{N}_{k=1}$ and $\vartriangle\subseteq\{1,2,\cdots,N\}$.
Find a matrix value function $M^{\vartriangle}$, admitting
\begin{itemize}
  \item $M^{\vartriangle}(x,t;z|\sigma_{d})$ is analytical in $\mathbb{C}\setminus(\mathcal{Z}\bigcup\mathcal{Z}^{*})$;
  \item $M^{\vartriangle}(x,t;z|\sigma_{d})=I+O(z^{-1})$, \quad $z\rightarrow\infty$;
  \item $M^{\vartriangle}(x,t;z|\sigma_{d})$ satisfies the following residue conditions at simple poles $z_{k}\in\mathcal{Z}$ and $z_{k}^{*}\in\mathcal{Z}^{*}$
\begin{align}
\begin{aligned}
&\mathop{Res}_{z=z_{k}}M^{\vartriangle}(x,t;z|\sigma_{d})=\mathop{lim}_{z\rightarrow z_{k}}M^{\vartriangle}(x,t;z|\sigma_{d})N^{\vartriangle}_{k},\\
&\mathop{Res}_{z=z_{k}^{*}}M^{\vartriangle}(x,t;z|\sigma_{d})=\mathop{lim}_{z\rightarrow z_{k}^{*}}M^{\vartriangle}(x,t;z|\sigma_{d})\sigma_{2}(N^{\vartriangle}_{k})^{*}\sigma_{2},
\end{aligned}
\end{align}
where
\begin{gather}
N_{k}^{\vartriangle}=\left\{
                                   \begin{aligned}
\left(
  \begin{array}{cc}
    0 & \gamma_{k}^{\vartriangle} \\
    0 & 0 \\
  \end{array}
\right),\quad k\in \vartriangle,\\
\left(
  \begin{array}{cc}
    0 & 0 \\
    \gamma_{k}^{\vartriangle} & 0 \\
  \end{array}
\right),\quad k\in \vartriangle,
\end{aligned}\right.~~\gamma_{k}^{\vartriangle}=\left\{
                                   \begin{aligned}
&c_{k}^{-1}(s_{11,\vartriangle}^{'}(z_{k}))^{-2}e^{-2it\theta(z_{k})}\quad k\in \vartriangle,\\
&c_{k}(s_{11,\vartriangle}(z_{k}))^{2}e^{2it\theta(z_{k})}\quad k\in \vartriangle,
\end{aligned}\right.\\
\theta(z_{k})=2z^{2}+\alpha z-\gamma+\frac{x}{t}z.\notag
\end{gather}
\end{itemize}
\end{RHP}
Since $M^{\vartriangle}(x,t;z|\sigma_{d})$ is directly transformed from $M(x,t;z|\sigma_{d})$, it is obvious that RHP \ref{RH-7} has unique solution.
\subsubsection{Long-time behavior of soliton solutions}
When $N=1$, the scattering data $\sigma_{d}=\{(z_{1}=\xi+i\eta, c_{1})\}$, we derive that
\begin{align}
u(x,t)=2\eta sech(2\eta\Omega(x,t))e^{-2i((2(\xi-\eta)^{2}-\alpha\xi+\gamma)t+\xi x)}e^{-i(\frac{\pi}{2}+arg(c_{1}))}e^{2i\int_{(-\infty,\infty)}^{(x,t)}\Delta},\\
v(x,t)=2\eta sech(2\eta\Omega(x,t))e^{2i((2(\xi-\eta)^{2}-\alpha\xi+\gamma)t+\xi x)}e^{i(\frac{\pi}{2}+arg(c_{1}))}e^{-2i\int_{(-\infty,\infty)}^{(x,t)}\Delta},
\end{align}
where $\Omega(x,t)=4\xi t+\alpha t-x-\frac{1}{2\eta}log(\frac{|c_{1}|}{2\eta})$. In fact, after the elastic collisions, the $N$-soliton asymptotically separate into $N$ single-soliton solutions as $t\rightarrow\infty$, of course, except the non-generic case which we do not discuss here such as two points of scattering data  lie on a vertical line.

Define  a distance
\begin{align}
\mu(I)=\min_{z_{k}\in Z\setminus Z(I)}\left\{Im(z_{k})dist(Rez_{k},I)\right\},
\end{align}
and a space-time cone
\begin{align}\label{space-time-S}
S(x_{1},x_{2},v_{1},v_{2})=\{(x,t),x=x_{0}+vt ~with ~x_{0}\in[x_{1},x_{2}],v\in[v_{1},v_{2}]\}.
\end{align}

\centerline{\begin{tikzpicture}[scale=0.8]
\path [fill=yellow] (-1,3)--(0,0) to (2,0) -- (3,3);
\path [fill=yellow] (-1,-3)--(0,0) to (2,0) -- (3,-3);
\draw[-][thick](-4,0)--(-3,0);
\draw[-][thick](-3,0)--(-2,0);
\draw[-][thick](-2,0)--(-1,0);
\draw[-][thick](-1,0)--(0,0);
\draw[-][thick](0,0)--(1,0);
\draw[-][thick](1,0)--(2,0);
\draw[-][thick](2,0)--(3,0);
\draw[-][thick](3,0)--(4,0);
\draw[->][thick](4,0)--(5,0)[thick]node[right]{$x$};
\draw[<-][thick](-2,3)[thick]node[right]{$t$}--(-2,2);
\draw[-][thick](-2,2)--(-2,1);
\draw[-][thick](-2,1)--(-2,0);
\draw[-][thick](-2,0)--(-2,-1);
\draw[-][thick](-2,-1)--(-2,-2);
\draw[-][thick](-2,-2)--(-2,-3);
\draw[fill] (0,0) circle [radius=0.08];
\draw[fill] (2,0) circle [radius=0.08];
\draw[fill] (-0.5,0)node[below]{$x_{2}$};
\draw[fill] (2.5,0)node[below]{$x_{1}$};
\draw[fill] (3.5,3)node[above]{$x=v_{2}t+x_{2}$};
\draw[fill] (3,-3)node[below]{$x=v_{1}t+x_{2}$};
\draw[fill] (-1,-3)node[below]{$x=v_{2}t+x_{1}$};
\draw[fill] (-2,3)node[above]{$x=v_{2}t+x_{1}$};
\draw[fill] (1,2)node[below]{$S$};
\draw[-][thick](-1,3)--(0,0);
\draw[-][thick](3,3)--(2,0);
\draw[-][thick](-1,-3)--(0,0);
\draw[-][thick](3,-3)--(2,0);
\end{tikzpicture}}
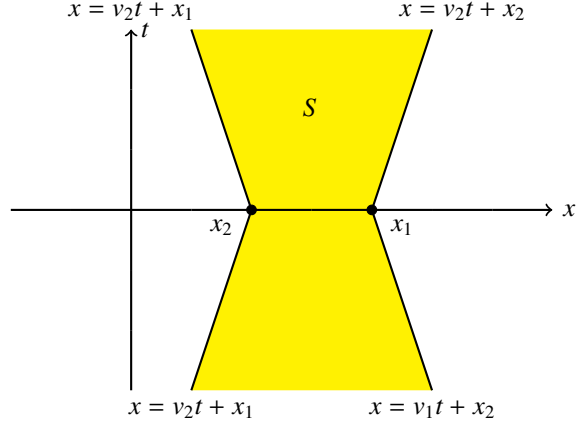
\centerline{\noindent {\small \textbf{Figure 4.} Space-time $S(x_{1},x_{2},v_{1},v_{2})$.}}
\begin{prop}\label{prop-6.3}
For given scattering data $\sigma_{d}^{\vartriangle_{z_{0}}^{-}}=\{(z_{k},\hat{c}_{k})\},$ for any fixed $x_{1},x_{2},v_{1},v_{2}\in R$ and $x_{1}<x_{2}$, $v_{1}<v_{2}$. Let $\mathcal{I}=\left[-\frac{v_{2}}{2},-\frac{v_{1}}{2}\right]$, then $t\rightarrow \infty$ and $(x,t)\in S(x_{1},x_{2},v_{1},v_{2})$, we have
\begin{align}\label{I-S}
M^{\vartriangle_{z_{0}}^{\mp}}(z|\sigma_{d}^{\vartriangle_{z_{0}}^{\pm}})=(I+O(e^{-8\mu t}))M^{\vartriangle_{z_{0}}^{\pm}(\mathcal{I})}
(z|\hat{\sigma}_{d}(\mathcal{I})),
\end{align}

where $M^{\vartriangle_{z_{0}}^{\mp}(\mathcal{I})}
(z|\hat{\sigma}_{d}(\mathcal{I}))$ is $N(I)=|\mathcal{Z}(\mathcal{I})|$-soliton solutions corresponding to scattering data
\begin{align}
\hat{\sigma}_{d}(\mathcal{I})=\{(z_{k},c_{k}(\mathcal{I})),z_{k}\in \mathcal{Z}(\mathcal{I})\},~
c_{k}(\mathcal{I})=c_{k}\prod_{z_{j}\in \mathcal{Z}\setminus \mathcal{Z}(\mathcal{I})}\left(\frac{z_{k}-z_{j}}{z_{k}-z^{*}_{j}}\right)^{2}.
\end{align}
\end{prop}
\centerline{\begin{tikzpicture}[scale=0.8]
\path [fill=yellow] (1.5,3)--(-1.5,3) to (-1.5,-3) -- (1.5,-3);
\draw[-][thick](-4,0)--(-3,0);
\draw[-][thick](-3,0)--(-2,0);
\draw[-][thick](-2,0)--(-1,0);
\draw[-][thick](-1,0)--(0,0);
\draw[-][thick](0,0)--(1,0);
\draw[-][thick](1,0)--(2,0);
\draw[-][thick](2,0)--(3,0);
\draw[->][thick](3,0)--(4,0)[thick]node[right]{$Rez$};
\draw[-][thick](-1.5,3)--(-1.5,2);
\draw[-][thick](-1.5,2)--(-1.5,1);
\draw[-][thick](-1.5,1)--(-1.5,0);
\draw[-][thick](-1.5,0)--(-1.5,-1);
\draw[-][thick](-1.5,-1)--(-1.5,-2);
\draw[-][thick](-1.5,-2)--(-1.5,-3);
\draw[-][thick](1.5,3)--(1.5,2);
\draw[-][thick](1.5,2)--(1.5,1);
\draw[-][thick](1.5,1)--(1.5,0);
\draw[-][thick](1.5,0)--(1.5,-1);
\draw[-][thick](1.5,-1)--(1.5,-2);
\draw[-][thick](1.5,-2)--(1.5,-3);
\draw[fill] (1.5,0)node[below]{$-\frac{v_{1}}{2}$} circle [radius=0.08];
\draw[fill] (-1.5,0)node[below]{$-\frac{v_{2}}{2}$} circle [radius=0.08];
\draw[fill] (3,1)node[below]{$z_{5}$} circle [radius=0.08];
\draw[fill] (3,-1)node[below]{$\bar{z}_{5}$} circle [radius=0.08];
\draw[fill] (2,3)node[below]{$z_{2}$} circle [radius=0.08];
\draw[fill] (2,-3)node[below]{$\bar{z}_{2}$} circle [radius=0.08];
\draw[fill] (0.5,2.5)node[below]{$z_{1}$} circle [radius=0.08];
\draw[fill] (0.5,-2.5)node[below]{$\bar{z}_{1}$} circle [radius=0.08];
\draw[fill] (-0.5,2.8)node[below]{$z_{3}$} circle [radius=0.08];
\draw[fill] (-0.5,-2.8)node[below]{$\bar{z}_{3}$} circle [radius=0.08];
\draw[fill] (-3,1.5)node[below]{$z_{4}$} circle [radius=0.08];
\draw[fill] (-3,-1.5)node[below]{$\bar{z}_{4}$} circle [radius=0.08];
\end{tikzpicture}}
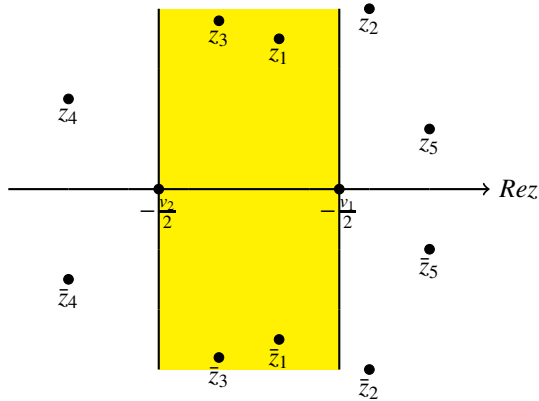
\centerline{\noindent {\small \textbf{Figure 5.} For fixed $v_{1}<v_{2}$, $I=\left[-\frac{v_{2}}{2},-\frac{v_{1}}{2}\right]$.}}

From RHP \ref{RH-7}, if we chose $\vartriangle=\vartriangle_{z_{0}}^{\mp}$, it is easy to check that
\begin{align}\label{Nk-estimate}
||N_{k}^{\vartriangle_{z_{0}}^{\mp}}||=\left\{
                                   \begin{aligned}
&o(1) \quad k\in \mathcal{Z}(\mathcal{I}),\\
&o(e^{-8\mu(\mathcal{I})t})\quad k\in \mathcal{Z}\setminus\mathcal{Z}(\mathcal{I}),
\end{aligned}\right.~~t\rightarrow\pm\infty,
\end{align}
which infers to the residues with $z_{k}\in\mathcal{Z}\setminus\mathcal{Z}(\mathcal{I})$ have little contribution to the solution $M^{\vartriangle_{z_{0}}^{\pm}}$.

Next, for each discrete spectrum point $z_{k}\in \mathcal{Z}\setminus \mathcal{Z}(\mathcal{I})$, we make a small disk $D_{k}$ whose radius is smaller than $\mu(\mathcal{I})$. Denoting that $\partial D_{k}$ is the boundary of $D_{k}$. Then, we introduce that
\begin{align}
\phi(z)=\left\{\begin{aligned}
&I-\frac{N_{k}^{\vartriangle^{-}(\mathcal{I})}}{z-z_{k}} \quad z\in D_{k},\\
&I+\frac{\sigma (N^{*}_{k})^{\vartriangle^{-}(\mathcal{I})\sigma}}{z-z_{k}} \quad z\in D^{*}_{k},\\
&I,\quad elsewhere,
\end{aligned}\right.
\end{align}
where
\begin{align*}
\vartriangle^{-}(\mathcal{I})=\{k:Rez_{k}<-\frac{v_{2}}{2}\},~~\vartriangle^{+}(\mathcal{I})
=\{k:Rez_{k}>-\frac{v_{1}}{2}\}.
\end{align*}
Making the transformation $\hat{M}^{\vartriangle_{z_{0}}^{\mp}}(z)=M^{\vartriangle_{z_{0}}^{\mp}}(z)\phi(z)$, we can derive that $\hat{M}^{\vartriangle_{z_{0}}^{\pm}}(z)$ has new jump in each boundary of the disk which can be denoted as $\hat{V}$, i.e.,
\begin{align}
\hat{M}^{\vartriangle_{z_{0}}^{\pm}}_{+}(z)=\hat{M}^{\vartriangle_{z_{0}}^{\pm}}_{-}(z)\hat{V},~~ z\in \partial D_{k}\cup D^{*}_{k}.
\end{align}
By using the estimate \eqref{Nk-estimate}, we have
\begin{align}
||\hat{V}-I||=o(e^{-8\mu(\mathcal{I})t}), ~~z\in \partial D_{k}\cup D^{*}_{k},~~ t\rightarrow\pm\infty.
\end{align}
It is obvious that $\hat{M}^{\vartriangle_{z_{0}}^{\pm}}(z|\sigma_{d})$ and $M^{\vartriangle_{z_{0}}^{\mp}}(z|\hat{\sigma}_{d})$ the same poles and residue conditions. Thus,
\begin{align}
\varepsilon(z)=\hat{M}^{\vartriangle_{z_{0}}^{\pm}}(z|\sigma_{d})
[M^{\vartriangle_{z_{0}}^{\mp}}(z|\hat{\sigma}_{d})]^{-1}
\end{align}
has no poles. However, there exists jump for $z\in\cup_{z_{k}\in\mathcal{Z}\setminus \mathcal{Z}(\mathcal{I})}(\partial D_{k}\cup \partial D^{*}_{k})$,
\begin{align*}
\varepsilon_{+}(z)=\varepsilon_{-}(z)V_{\varepsilon},
\end{align*}
where $V_{\varepsilon}=M^{\vartriangle_{z_{0}}^{\mp}}\hat{V}(^{\vartriangle_{z_{0}}^{\mp}})^{-1}$ which implies that
\begin{align}
||V_{\varepsilon}-I||=o(e^{-8\mu(\mathcal{I})t}), ~~t\rightarrow\pm\infty.
\end{align}
Then, using the theory of small-norm Riemann-Hilbert problems, one can easily derive that
\begin{align*}
\varepsilon(z)=I+o(e^{-8\mu(\mathcal{I})t}), ~~t\rightarrow\pm\infty,
\end{align*}
which together with $\hat{M}^{\vartriangle_{z_{0}}^{\mp}}(z)=M^{\vartriangle_{z_{0}}^{\mp}}(z)\phi(z)$ gives the formula \eqref{I-S}.
\subsection{Local solvable model near phase point $z=z_{0}=-\frac{1}{4}(\frac{x}{t}+\alpha)$}
From \eqref{V2-I}, we can easily find that $V^{(2)}-I$ does not have a uniform estimate for large time. Therefore, we use a local model $M^{(out)}(z)M^{(pc)}(z_0,r_0)$ to match the the jumps of $M_{RHP}$ on $\Sigma^{(2)}\cap \mathcal{U}_{z_{0}}$ to make the jump uniformly for the function $E(z)$. Recall the definition of $\theta(z)$ \eqref{theta} and introduce the transformation
\begin{align}\label{6.4}
\lambda=\lambda(z)=2\sqrt{2t}(z-z_{0}).
\end{align}
Then, we can derive that
\begin{align}
2t\theta=\frac{1}{2}\lambda^{2}-4tz_{0}^{2}-2\gamma t,
\end{align}
which means $\mathcal{U}_{z_0}$ can be mapped into an expanding neighborhood of $\lambda=0$. If we let
\begin{align}
r_0(z_0)=r(z_0)T_0(z_0)^{-2}e^{2i(\nu(z_0)log(2\sqrt{2t}))}e^{-4itz_0^{2}-2it\gamma}
\end{align}
and consider the fact that $1-\chi_{\mathcal{Z}}=1$ as $z\in\mathcal{U}_{z_0}$, the jump of $M_{RHP}$ in $\mathcal{U}_{z_0}$ can be expressed as
\begin{align}
V^{(2)}(z)\mid_{z\in\mathcal{U}_{z_0}}=\left\{\begin{aligned}
\lambda(z)^{i\nu\hat{\sigma}_{3}e^{-\frac{i\lambda(z)^{2}}{4}
\hat{\sigma}_{3}}}\left(
                    \begin{array}{cc}
                      1 & 0 \\
                      r_{0}(z_0) & 1 \\
                    \end{array}
                  \right),\quad z\in\Sigma_{1},\\
\lambda(z)^{i\nu\hat{\sigma}_{3}e^{-\frac{i\lambda(z)^{2}}{4}
\hat{\sigma}_{3}}}\left(
                    \begin{array}{cc}
                      1 & \frac{r^{*}_{0}(z_0)}{1+|r_{0}(z_0)|^{2}} \\
                      0 & 1 \\
                    \end{array}
                  \right),\quad z\in\Sigma_{2},\\
\lambda(z)^{i\nu\hat{\sigma}_{3}e^{-\frac{i\lambda(z)^{2}}{4}
\hat{\sigma}_{3}}}\left(
                    \begin{array}{cc}
                      1 & 0\\
                      \frac{r_{0}(z_0)}{1+|r_{0}(z_0)|^{2}} & 1 \\
                    \end{array}
                  \right),\quad z\in\Sigma_{3},\\
\lambda(z)^{i\nu\hat{\sigma}_{3}e^{-\frac{i\lambda(z)^{2}}{4}
\hat{\sigma}_{3}}}\left(
                    \begin{array}{cc}
                      1 & r^{*}_{0}(z_0) \\
                      0 & 1 \\
                    \end{array}
                  \right),\quad z\in\Sigma_{4}.
\end{aligned}\right.
\end{align}
Obviously, the above jump $V^{(2)}(z)\mid_{z\in\mathcal{U}_{z_0}}$ is equivalent to the jump of the parabolic cylinder model problem\eqref{Vpc} which we have obtain the solutions in appendix $A$. Then, since $M^{(out)}(z)$ an analytic and bounded function in the $\mathcal{U}_{z_0}$,  $M^{(out)}(z)M^{(pc)}(z_0,r_0)$ which has been defined in \eqref{Mrhp} admits the jump $V^{(2)}(z)$ of $M_{RHP}(z)$.

\subsection{The small-norm RHP for $E(z)$}
Recall the transform \eqref{Mrhp}, we can derive that
\begin{align}\label{explict-E(z)}
E(z)=\left\{\begin{aligned}
&M_{RHP}(z)M^{(out)}(z)^{-1}, &&z\in\mathbb{C}\backslash \mathcal{U}_{z_0},\\
&M_{RHP}(z)M^{(pc)}(z_0,r_0)^{-1}M^{(out)}(z)^{-1}, &&z\in \mathcal{U}_{z_0},
\end{aligned} \right.
\end{align}
which is analytic in $\mathbb{C}\setminus\Sigma^{(E)}$
where
\begin{align}
\Sigma^{(E)}=\partial\mathcal{U}_{z_0}\bigcup(\Sigma^{(2)}\backslash\mathcal{U}_{z_0}),
\end{align}
with clockwise direction for $\partial\mathcal{U}_{z_0}$.
Then we can check that $E(z)$ satisfies the Riemann-Hilbert problem.
\begin{RHP}\label{RH-8}
Find a matrix-valued function $E(z)$ such that
\begin{itemize}
 \item $E$ is analytical in $\mathbb{C}\backslash \Sigma^{(E)}$;\\
 \item $E(z)=I+O(z^{-1})$, \quad $z\rightarrow\infty$; \\
 \item $E_+(z)=E_-(z)V^{(E)}(z)$, \quad $z\in\Sigma^{(E)}$, where
 \end{itemize}
 \begin{align}\label{6.3}
 V^{(E)}(z)=\left\{\begin{aligned}
 &M^{(out)}(z)V^{(2)}(z)M^{(out)}(z)^{-1}, &&z\in\Sigma^{(2)}\backslash \mathcal{U}_{z_0},\\
 &M^{(out)}(z)M^{(pc)}(\xi,r_0)M^{(out)}(z)^{-1}, &&z\in\partial\mathcal{U}_{z_0}.
 \end{aligned}\right.
 \end{align}
\end{RHP}
\centerline{\begin{tikzpicture}[scale=0.6]
\draw(0,0) [black, line width=1] circle(2);
\draw[-][thick](1.414,1.414)--(4,4);
\draw[-][thick](-1.414,1.414)--(-4,4);
\draw[-][thick](-1.414,-1.414)--(-4,-4);
\draw[-][thick](1.414,-1.414)--(4,-4);
\draw[->][thick](2,2)--(3,3);
\draw[->][thick](-4,4)--(-3,3);
\draw[->][thick](-4,-4)--(-3,-3);
\draw[->][thick](2,-2)--(3,-3);
\draw[fill] (3,1)node[below]{$\partial \mathcal {U}_{z_{0}}$};
\draw[fill] (2,3)node[above]{$\Sigma^{(2)}\setminus\mathcal {U}_{z_{0}}$};
\draw[->][black, line width=0.8] (2,0) arc(0:-270:2);
\end{tikzpicture}}
\centerline{\noindent {\small \textbf{Figure 6.} Jump contour $\Sigma^{(E)}=\partial \mathcal {U}_{z_{0}}\cup(\Sigma^{(2)}\setminus\mathcal {U}_{z_{0}})$}.}

Base on the estimate $V^{(2)}-I$ shown in \eqref{V2-I} and the boundedness of $M^{(out)}$, it is easy to check that
\begin{align}\label{VE-I}
|V^{(E)}(z)-I|=\left\{\begin{aligned}
&\mathcal{O}(e^{-2t|z-z_0|^2}) &&z\in\Sigma^{(2)}\backslash\mathcal{U}_{z_0},\\
&\mathcal{O}(t^{-1/2}) &&z\in\partial\mathcal{U}_{z_0},
\end{aligned}\right.
\end{align}
then, we can derive that
\begin{align}\label{6.1}
||(z-z_0)^{k}(V^{(E)}-I)||_{L^{p}(\Sigma^{(E)})}=o(t^{-1/2}),~~p\in[1,+\infty],~k\geq0.
\end{align}
As a small-norm Riemann-Hilbert problem, the existence and uniqueness of RHP \ref{RH-8} are guaranteed and we obtain that
\begin{align}\label{E(z)-solution}
E(z)=I+\frac{1}{2\pi i}\int_{\Sigma_E}\frac{\mu_E(s)(V^{(E)}(s)-I)}{s-z}ds
\end{align}
where $\mu_E\in L^2 (\Sigma^{(E)}) $ and satisfies
\begin{align}\label{6.2}
(1-C_{\omega_E})\mu_E=I,
\end{align}
where $C_{\omega_E}$ is an integral operator which is defined by
\begin{align*}
C_{\omega_E}f=C_{-}(f(V^{(E)}-I)),\\
C_{-}f(z)\lim_{z\rightarrow\Sigma_{-}^{(E)}}\int_{\Sigma_E}\frac{f(s)}{s-z}ds,
\end{align*}
where $C_{-}$ is the Cauchy projection operator. Then, according to the properties of the Cauchy projection operator $C_{-}$, and the estimate of $V^{(E)}-I$ shown in \eqref{VE-I}, we can derive that
\begin{align}
\|C_{\omega_E}\|_{L^2(\Sigma^{(E)})}\lesssim\|C_-\|_{L^2(\Sigma^{(E)})}\|V^{(E)}-I\|_{L^{\infty}
(\Sigma^{(E)})}\lesssim\mathcal{O}(t^{-1/2}).
\end{align}
which implies that $1-C_{\omega_E}$ is invertible and thus the existence and uniqueness of $\mu_E$ are guaranteed, and thus of $E(z)$. Now, it can be explained that the definition of $M_{RHP}$ is reasonable, and in turn we can solve \eqref{delate-pure-RHP} to the unknown $M^{(3)}$ which admits the pure $\bar{\partial}$-Problem\eqref{RH-4}.

Then, to reconstruct the solutions of $u(x,t)$ and $v(x,t)$, the asymptotic behavior of $E(z)$ as $z\rightarrow\infty$ is needed. According to the expression \eqref{E(z)-solution} of $E(z)$, we obtain that
\begin{align}
E(z)=I+z^{-1}E_{1}+\mathcal{O}(z^{-2}),\quad z\rightarrow\infty,
\end{align}
where
\begin{align}
E_1(x,t)=-\frac{1}{2i\pi}\int_{\Sigma_E}\mu_E(s)(V^{(E)}(s)-I)ds.
\end{align}
Then, by applying \eqref{6.2} and the estimate \eqref{VE-I}, \eqref{6.1}, we obtain that
\begin{align}
E_1(x,t)=-\frac{1}{2i\pi}\oint_{\partial\mathcal{U}_{z_{0}}}(V^{(E)}(s)-I)ds+o(t^{-1}).
\end{align}
Finally, applying \eqref{6.3}, \eqref{6.4} and \eqref{A-1} shown in Appendix $A$, we derive that
\begin{align}\label{6.5}
\begin{split}
E_1(x,t)&=\frac{1}{2i\sqrt{2t}}M^{(out)}(z_0)M_1^{(pc)}(z_0)M^{(out)}(z_0)^{-1}+\mathcal{O}(t^{-1})\\
&=\frac{1}{2i\sqrt{2t}}M^{(out)}(z_0)\begin{pmatrix}0&\beta_{12}\\-\beta_{21}&0\end{pmatrix}
M^{(out)}(z_0)^{-1}+\mathcal{O}(t^{-1}),
\end{split}
\end{align}
with
\begin{align}\label{6.6}
\beta_{12}(z_0)=\beta_{21}^{*}(z_0)=\tau(z_0,+)e^{ix^{2}/(4t)+i\alpha^{2}t/(4)-2\alpha x-i\nu(z_{0})\log |8t|}.
\end{align}
Here, $|\tau(z_0,+)|^{2}=|\nu(z_{0})^{2}|$ and
\begin{align*}
\arg\tau(z_{0},+)=\frac{\pi}{4}+\arg\Gamma(i\nu(z_{0}))-\arg r(z_{0})- 2\int^{z_{0}}_{-\infty}\ln|z-z_{0}|\mathrm{d}\nu(s).
\end{align*}

\section{Pure $\bar{\partial}$-Problem}
In this section, we analysis the remaining $\bar{\partial}$-Problem. The $\bar{\partial}$-RH problem \ref{RH-4} for $M^{(3)}(z)$ is equivalent to the following integral equation
\begin{align}\label{7.1}
M^{(3)}(z)=\mathrm{I}-\frac{1}{\pi}\iint_{\mathbb{C}}\frac{M^{(3)}W^{(3)}}{s-z}\mathrm{d}A(s),
\end{align}
where $\mathrm{d}A(s)$ is Lebesgue measure. Further, we write the equation \eqref{7.1} in operator form
\begin{align}\label{7.2}
(\mathrm{I}-\mathrm{S})M^{(3)}(z)=\mathrm{I},
\end{align}
where $\mathrm{S}$ is Cauchy operator
\begin{align}\label{7.3}
\mathrm{S}[f](z)=-\frac{1}{\pi}\iint_{\mathbb{C}}\frac{f(s)W^{(3)}(s)}{s-z}\mathrm{d}A(s).
\end{align}
We need to prove that the inverse operator $(\mathrm{I}-\mathrm{S})^{-1}$ exists, so that the solution $M^{(3)}(z)$ exists.

\begin{lem}
For large $t$, there exists a constant $c$ that makes the operator \eqref{7.3} admits the following relation
\begin{align}\label{7.4}
||\mathrm{S}||_{L^{\infty}\rightarrow L^{\infty}}\leq ct^{-1/4}.
\end{align}
\end{lem}
\begin{proof}
We mainly prove the case that the matrix function supported in the region $\Omega_1$, the other case can be proved similarly. Denoted that $f\in L^{\infty}(\Omega_1)$ and $s=p+iq$, then based on \eqref{4.1} and \eqref{5.1}, we can obtain the following inequality
\begin{align}\label{7.5}
|S[f](z)|&\leq\frac{1}{\pi}\iint_{\Omega_{1}}\frac{|f(s)M_{RHP}(s)\bar{\partial}R_{1}(s)M_{RHP}(s)^{-1}|}
{|s-z|}df(s)\\
&\leq\big|\big|f\big|\big|_{L^{\infty}(\Omega_{1})}\big|\big|M_{RHP}\big|\big|_{L^{\infty}(\hat{\Omega}_{1})}
\big|\big|M_{RHP}^{-1}\big|\big|_{L^{\infty}(\hat{\Omega}_{1})}\frac{1}{\pi}\iint_{\Omega_{1}}
\frac{|\bar{\partial}R_{1}(s)||e^{-8tq(p-z_{0})}|}{|s-z|}df(s),
\end{align}
where $\hat{\Omega}_{1}\triangleq\Omega_1\cap supp(1-\chi_{\mathcal{Z}})$ is bounded far from the scattering data $z_{k}$ of $M_{RHP}$, and thus $||M_{RHP}^{\pm1}||_{L^{\infty}(\hat{\Omega}_{1})}||$ are finite.

Based on \eqref{R-estimate} and the estimates shown in Appendix $B$, from \eqref{7.5}, we obtain that
\begin{align}\label{7.6}
||\mathrm{S}||_{L^{\infty}\rightarrow L^{\infty}}\leq c(I_{1}+I_{2}+I_{3})\leq ct^{-1/4},
\end{align}
where
\begin{align}\label{7.8}
I_{1}=\iint_{\Omega_{1}}
\frac{|\bar{\partial}\chi_{\mathcal{Z}}(s)||e^{-8tq(p-z_{0})}|}{|s-z|}df(s), ~~
I_{2}=\iint_{\Omega_{1}}
\frac{|r'(p)||e^{-8tq(p-z_{0})}|}{|s-z|}df(s),
\end{align}
and
\begin{align}\label{7.9}
I_{3}=\iint_{\Omega_{1}}
\frac{|s-z_{0}|^{-\frac{1}{2}}|e^{-8tq(p-z_{0})}|}{|s-z|}df(s).
\end{align}
\end{proof}

Next, our purpose is to recover the large time asymptotic behaviors of $u(x,t)$ and $v(x,t)$. According to \eqref{2.13}, we need determine the coefficient of the $z^{-1}$ term of $M^{(3)}$ in the Laurent expansion at infinity. Base on \eqref{7.1}, we can derive that
\begin{align*}
M^{(3)}(z)=\mathrm{I}-\frac{1}{\pi}\int\int_{\mathbb{C}}\frac{M^{(3)}(s)W^{(3)}(3)}
{s-z}\mathrm{d}A(s)
=\mathrm{I}+\frac{M^{(3)}_{1}}{z}+o(z^{-2}),
\end{align*}
where
\begin{align*}
M^{(3)}_{1}=\frac{1}{\pi}\int\int_{\mathbb{C}}M^{(3)}(s)W^{(3)}(s)
\mathrm{d}A(s).
\end{align*}
In what follows, the the coefficient of the $z^{-1}$ term of $M^{(3)}$ in the Laurent expansion at infinity, i.e., $M^{(3)}_{1}$ satisfies the following property.
\begin{lem}\label{prop-7.2}
For large $t$, there exists a constant $c$ that makes $M^{(3)}_{1}$ admits the following inequality
\begin{align}\label{7.7}
M^{(3)}_{1}\leq ct^{-3/4}.
\end{align}
\end{lem}
The proof of this Lemma is similar to the process that shown in  Appendix $B$.

\section{Soliton resolution for the cgNLS equation}
Now, we are going to construct the long time asymptotic of the cgNLS equation \label(1.1).
Recall a series of transformation process including \eqref{Trans-1}, \eqref{Trans-2}, \eqref{delate-pure-RHP} and \eqref{Mrhp}, i.e.,
\begin{align*}
M(z)\leftrightarrows M^{(1)}(z)\leftrightarrows M^{(2)}(z)\leftrightarrows M^{(3)}(z) \leftrightarrows E(z),
\end{align*}
we then obtain
\begin{align*}
M(z)=M^{(3)}(z)E(z)M^{(out)}(z)R^{(2)^{-1}}(z)T^{\sigma_{3}}(z),~~ z\in\mathbb{C}\setminus\mathcal{U}_{z_{0}}.
\end{align*}
In order to recover the solution $u(x,t)$ and $v(x,t)$, we take $z\rightarrow\infty$ vertically which implies $z\in\Omega_{2}$ or $z\in\Omega_{5}$, thus $R^{(2)}(z)=I$. Then, we have
\begin{align*}
M=\left(\mathrm{I}+\frac{M^{(3)}_{1}}{z}+\cdots\right)\left(\mathrm{I}+\frac{E_{1}}{z}+\cdots\right)
\left(\mathrm{I}+\frac{M_{1}^{(out)}}{z}+\cdots\right)\left(\mathrm{I}
+\frac{T_{1}\sigma_{3}}{z}+\cdots\right),
\end{align*}
from which we can derive that
\begin{align*}
M_{1}=M_{1}^{(out)}+E_{1}+M_{1}^{(3)}+T_{1}\sigma_{3}.
\end{align*}
Then, according to the reconstruction formula \eqref{2.13} and the Lemma \ref{prop-7.2}, we obtain that
\begin{align}\label{8.1}
u(x,t)e^{-2i\int_{(-\infty,\infty)}^{(x,t)}\Delta}=2i(M_{1}^{(out)})_{12}+2i(E_{1})_{12}+O(t^{-3/4}),\\
v(x,t)e^{2i\int_{(-\infty,\infty)}^{(x,t)}\Delta}=2i(M_{1}^{(out)})_{21}+2i(E_{1})_{21}+O(t^{-3/4}).
\end{align}
Based on the Corollary \ref{prop-6.2} and the formulae \eqref{6.5}, \eqref{6.6}, we obtain the soliton resolution
\begin{align*}
u(x,t)e^{-2i\int_{(-\infty,\infty)}^{(x,t)}\Delta}=
u_{sol}(x,t;\sigma_{d}^{out})e^{-2i\int_{(-\infty,\infty)}^{(x,t)}\Delta}
+t^{-\frac{1}{2}}f^{+}_{1}+O(t^{-\frac{3}{4}}),\\
v(x,t)e^{2i\int_{(-\infty,\infty)}^{(x,t)}\Delta}
=v_{sol}(x,t;\sigma_{d}^{out})e^{-2i\int_{(-\infty,\infty)}^{(x,t)}\Delta}
+t^{-\frac{1}{2}}f^{+}_{2}+O(t^{-\frac{3}{4}}).
\end{align*}
where
\begin{align*}
f^{+}_{1}=\frac{1}{\sqrt{2}}m_{11}^{2}\tau(z_{0},+)e^{\theta_{1}(x,t)}+\frac{1}{2\sqrt{2}}m_{12}^{2}\tau^{*}(z_{0},+)e^{\theta_{2}(x,t)},\\
f^{+}_{2}=-\frac{1}{\sqrt{2}}m_{22}^{2}\tau^{*}(z_{0},+)e^{\theta_{2}(x,t)}-\frac{1}{2\sqrt{2}}m_{21}^{2}\tau(z_{0},+)e^{\theta_{1}(x,t)},
\end{align*}
with $z_{0}=-\frac{1}{4}(\frac{x}{t}-\alpha)$,
\begin{align*}
\theta_{1}(x,t)=\frac{ix^{2}}{4t}+i\frac{\alpha^{2}t}{4}-2\alpha x-i\nu(z_{0})\log |8t|,\\
\theta_{2}(x,t)=-\frac{ix^{2}}{4t}-i\frac{\alpha^{2}t}{4}-2\alpha x+i\nu(z_{0})\log |8t|,
\end{align*}
and $m_{11}$, $m_{12}$, $m_{21}$ and $m_{22}$ come from the elements of the matrix solution of RHP \ref{RH-6}.
Via using Proposition \ref{prop-6.3}, there exists exponential errors which are absorbed into the $O(t^{-3/4})$ term when we replace $u_{sol}(x,t;\sigma_{d}^{out})$ and $v_{sol}(x,t;\sigma_{d}^{out})$ with $u_{sol}(x,t;\hat{\sigma}_{d}(I))$ and $v_{sol}(x,t;\hat{\sigma}_{d}(I))$, respectively. Then, we consider the term $e^{-2i\int_{(-\infty,\infty)}^{(x,t)}\Delta}$. As $t\rightarrow \infty$, we have
\begin{align*}
e^{-2i\int_{(-\infty,\infty)}^{(x,t)}\Delta}=&e^{-2i\beta\int_{-\infty}^{x}u(s,t)v(s,t)ds}\\
=&e^{-2i\beta\int_{(-\infty,\infty)}^{(x,t)}\mathcal{M}(s,t)ds}+O(t^{-\frac{3}{2}}),
\end{align*}
where
\begin{align}\label{8.3}
\mathcal{M}^{+}(x,t)=|u_{sol}(x,t;\hat{\sigma}_{d}(I))v_{sol}(x,t;\hat{\sigma}_{d}(I)) +t^{-1}f^{+}_{1}f^{+}_{2}|.
\end{align}
Finally, we can give the more general solution of $u(x,t)$ and $v(x,t)$ and summarise as the following theorem.
\begin{thm}\label{prop-1.3}
With the initial value conditions that $u_{0}(x), v_{0}(x)\in H^{1,1}(\mathbb{R})$ and the assumption that the initial values $u_{0}(x), v_{0}(x)$ satisfy the Assumption \eqref{assum}, let $u(x,t)$ and $v(x,t)$ be the solution of cgNLS equations \eqref{1.1}. The scattering data denoted as $\{r,\{z_{k},c_{k}\}_{k=1}^{N}\}$ which generated from the initial values $u_{0}(x), v_{0}(x)$. For fixed $x_{1},x_{2},v_{1},v_{2}\in \mathbb{R}$ with $x_{1}<x_{2}$, $v_{1}<v_{2}$, and $\mathcal{I}=\left[-\frac{v_{2}}{2},-\frac{v_{1}}{2}\right]$, $z_{0}=-\frac{1}{4}(\frac{x}{t}+\alpha)$, then as $t\rightarrow \infty$ and $(x,t)\in S(x_{1},x_{2},v_{1},v_{2})$ which is defined in \eqref{space-time-S}, the solution $u(x,t)$ and $v(x,t)$ can be expressed as
\begin{align}\label{8.2}
\begin{split}
u(x,t)=(u_{sol}(x,t;\hat{\sigma}_{d}(I))+t^{-\frac{1}{2}}f^{\pm}_{1}) e^{2i\beta\int_{(-\infty,\infty)}^{(x,t)}\mathcal{M}(s,t)ds}+O(t^{-\frac{3}{4}}),\\
v(x,t)
=(v_{sol}(x,t;\hat{\sigma}_{d}(I))+t^{-\frac{1}{2}}f^{\pm}_{2}) e^{-2i\beta\int_{(-\infty,\infty)}^{(x,t)}\mathcal{M}(s,t)ds}+O(t^{-\frac{3}{4}}).
\end{split}
\end{align}
Here, $u_{sol}(x,t;\hat{\sigma}_{d}(I))$ and $v_{sol}(x,t;\hat{\sigma}_{d}(I))$ are the $N(\mathcal{I})=|\mathcal{Z}(\mathcal{I})|$ soliton, and
\begin{align*}
\mathcal{M}^{\pm}(x,t)=|u_{sol}(x,t;\hat{\sigma}_{d}(I))v_{sol}(x,t;\hat{\sigma}_{d}(I)) +t^{-1}f^{\pm}_{1}f^{\pm}_{2}|,\\
f^{\pm}_{1}=\frac{1}{\sqrt{2}}m_{11}^{2}\tau(z_{0},\pm)e^{\theta^{\mp}_{1}(x,t)}
+\frac{1}{\sqrt{2}}m_{12}^{2}\tau^{*}(z_{0},\pm)e^{\theta^{\pm}_{2}(x,t)},\\
f^{\pm}_{2}=-\frac{1}{\sqrt{2}}m_{22}^{2}\tau^{*}(z_{0},\pm)e^{\theta^{\pm}_{2}(x,t)}
-\frac{1}{\sqrt{2}}m_{21}^{2}\tau(z_{0},\pm)e^{\theta^{\mp}_{1}(x,t)},
\end{align*}
with $|\tau(z_0,\pm)|^{2}=|\nu(z_{0})^{2}|$,
\begin{gather*}
\arg\tau(z_{0},\pm)=\frac{\pi}{4}+\arg\Gamma(i\nu(z_{0}))-\arg r(z_{0})\mp 2\int^{z_{0}}_{-\infty}\ln|z-z_{0}|\mathrm{d}\nu(s),\\
\theta^{\mp}_{1}(x,t)=\frac{ix^{2}}{4t}+i\frac{\alpha^{2}t}{4}-2\alpha x\mp i\nu(z_{0})\log |8t|,\\
\theta^{\pm}_{2}(x,t)=-\frac{ix^{2}}{4t}-i\frac{\alpha^{2}t}{4}-2\alpha x\pm i\nu(z_{0})\log |8t|,
\end{gather*}
and $m_{11}$, $m_{12}$, $m_{21}$ and $m_{22}$ come from the elements of the matrix solution of RHP\ref{RH-6}.
\end{thm}

\begin{rem}
The results shown in Theorem \ref{prop-1.3} need the condition that the initial values $u_{0}(x), v_{0}(x) \in H^{1,1}(\mathbb{R})$, then the inverse scattering transform possesses  well mapping properties, and it guarantees the condition that there exists no discrete spectrum on the real axis. It is noted that the asymptotic results only depend on the $H^{1}(\mathbb{R})$ norm of $r$, therefore, for any $u_{0}(x), v_{0}(x)\in H^{1,1}(\mathbb{R})$ admitting the Assumption \eqref{assum}, the process of the large-time analysis and calculations shown in this work is unchanged.
\end{rem}



\section*{Acknowledgements}

This work was supported by  the National Natural Science Foundation of China under Grant No. 11975306, the Natural Science Foundation of Jiangsu Province under Grant No. BK20181351, the Six Talent Peaks Project in Jiangsu Province under Grant No. JY-059,  and the Fundamental Research Fund for the Central Universities under the Grant Nos. 2019ZDPY07 and 2019QNA35.

\section*{Appendix A: The parabolic cylinder model problem}
Here, we describe the solution of  parabolic cylinder model problem\cite{PC-model,PC-model-2}.
Define the contour $\Sigma^{pc}=\cup_{j=1}^{4}\Sigma_{j}^{pc}$ where
\begin{align}
\Sigma_{j}^{pc}=\left\{\lambda\in\mathbb{C}|\arg\lambda=\frac{2j-1}{4}\pi \right\}.\tag{A.1}
\end{align}
For $r_{0}\in \mathbb{C}$, let $\nu(r)=-\frac{1}{2\pi}\log(1+|r_{0}|^{2})$, consider the following parabolic cylinder model Riemann-Hilbert problem.
\begin{RHP}\label{PC-model}
Find a matrix-valued function $M^{(pc)}(\lambda)$ such that
\begin{align}
&\bullet \quad M^{(pc)}(\lambda)~ \text{is analytic in}~ \mathbb{C}\setminus\Sigma^{pc}, \tag{A.2}\\
&\bullet \quad M_{+}^{(pc)}(\lambda)=M_{-}^{(pc)}(\lambda)V^{(pc)}(\lambda),\quad
\lambda\in\Sigma^{pc}, \tag{A.3}\\
&\bullet \quad M^{(pc)}(\lambda)=I+\frac{M_{1}}{\lambda}+O(\lambda^{2}),\quad
\lambda\rightarrow\infty. \tag{A.4}
\end{align}
where
\begin{align}\label{Vpc}
V^{(pc)}(\lambda)=\left\{\begin{aligned}
\lambda^{i\nu\hat{\sigma}_{3}e^{-\frac{i\lambda^{2}}{4}
\hat{\sigma}_{3}}}\left(
                    \begin{array}{cc}
                      1 & 0 \\
                      r_{0} & 1 \\
                    \end{array}
                  \right),\quad \lambda\in\Sigma_{1}^{pc},\\
\lambda^{i\nu\hat{\sigma}_{3}e^{-\frac{i\lambda^{2}}{4}
\hat{\sigma}_{3}}}\left(
                    \begin{array}{cc}
                      1 & \frac{r^{*}_{0}}{1+|r_{0}|^{2}} \\
                      0 & 1 \\
                    \end{array}
                  \right),\quad \lambda\in\Sigma_{2}^{pc},\\
\lambda^{i\nu\hat{\sigma}_{3}e^{-\frac{i\lambda^{2}}{4}
\hat{\sigma}_{3}}}\left(
                    \begin{array}{cc}
                      1 & 0\\
                      \frac{r_{0}}{1+|r_{0}|^{2}} & 1 \\
                    \end{array}
                  \right),\quad \lambda\in\Sigma_{3}^{pc},\\
\lambda^{i\nu\hat{\sigma}_{3}e^{-\frac{i\lambda^{2}}{4}
\hat{\sigma}_{3}}}\left(
                    \begin{array}{cc}
                      1 & r^{*}_{0} \\
                      0 & 1 \\
                    \end{array}
                  \right),\quad \lambda\in\Sigma_{4}^{pc},
\end{aligned}\right.\tag{A.5}
\end{align}
\end{RHP}

\centerline{\begin{tikzpicture}[scale=0.6]
\draw[-][dashed](-4,0)--(-3,0);
\draw[-][dashed](-3,0)--(-2,0);
\draw[-][dashed](-2,0)--(-1,0);
\draw[-][dashed](-1,0)--(0,0);
\draw[-][dashed](0,0)--(1,0);
\draw[-][dashed](1,0)--(2,0);
\draw[-][dashed](2,0)--(3,0);
\draw[-][dashed](3,0)--(4,0);
\draw[-][thick](-4,-4)--(4,4);
\draw[-][thick](-4,4)--(4,-4);
\draw[->][thick](2,2)--(3,3);
\draw[->][thick](-4,4)--(-3,3);
\draw[->][thick](-4,-4)--(-3,-3);
\draw[->][thick](2,-2)--(3,-3);
\draw[fill] (3.2,3)node[below]{$\Sigma_{1}^{pc}$};
\draw[fill] (3.2,-3)node[above]{$\Sigma_{4}^{pc}$};
\draw[fill] (-3.2,3)node[below]{$\Sigma_{2}^{pc}$};
\draw[fill] (-2,-3)node[below]{$\Sigma_{3}^{pc}$};
\draw[fill] (0,0)node[below]{$0$};
\draw[fill] (1,0)node[below]{$\Omega_{6}$};
\draw[fill] (1,0)node[above]{$\Omega_{1}$};
\draw[fill] (0,-1)node[below]{$\Omega_{5}$};
\draw[fill] (0,1)node[above]{$\Omega_{2}$};
\draw[fill] (-1,0)node[below]{$\Omega_{4}$};
\draw[fill] (-1,0)node[above]{$\Omega_{3}$};
\draw[fill] (7,3)node[below]{$\lambda^{i\nu\hat{\sigma}_{3}}e^{-\frac{i\lambda^{2}}{4}\hat{\sigma}_{3}}
\left(
  \begin{array}{cc}
    1 & 0 \\
    r_{0} & 1 \\
  \end{array}
\right)
$};
\draw[fill] (7,-2)node[below]{$\lambda^{i\nu\hat{\sigma}_{3}}e^{-\frac{i\lambda^{2}}{4}\hat{\sigma}_{3}}
\left(
  \begin{array}{cc}
    1 & r^{*}_{0} \\
    0 & 1 \\
  \end{array}
\right)
$};
\draw[fill] (-7,2.5)node[below]{$\lambda^{i\nu\hat{\sigma}_{3}}e^{-\frac{i\lambda^{2}}{4}\hat{\sigma}_{3}}
\left(
  \begin{array}{cc}
    1 & \frac{r^{*}_{0}}{1+|r_{0}|^{2}} \\
    0 & 1 \\
  \end{array}
\right)
$};
\draw[fill] (-7,-1)node[below]{$\lambda^{i\nu\hat{\sigma}_{3}}e^{-\frac{i\lambda^{2}}{4}\hat{\sigma}_{3}}
\left(
  \begin{array}{cc}
    1 & 0 \\
    \frac{r_{0}}{1+|r_{0}|^{2}} & 1 \\
  \end{array}
\right)
$};
\end{tikzpicture}}
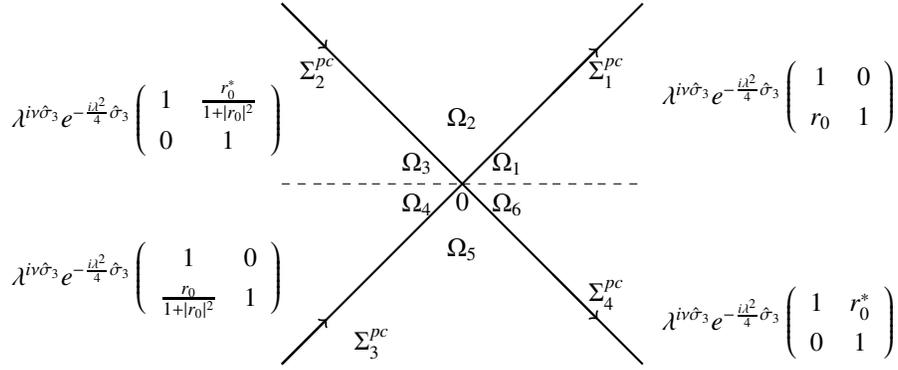
\centerline{\noindent {\small \textbf{Figure 7.} Jump matrix $V^{(pc)}$}.}

We know that the parabolic cylinder equation can be expressed as \cite{PC-equation}
\begin{align*}
\left(\frac{\partial^{2}}{\partial z}+(\frac{1}{2}-\frac{z^{2}}{2}+a)\right)D_{a}=0.
\end{align*}
As shown in the literature\cite{Deift-1993, PC-solution2}, we obtain the explicit solution $M^{(pc)}(\lambda, r_{0})$:
\begin{align*}
M^{(pc)}(\lambda, r_{0})=\Phi(\lambda, r_{0})\mathcal{P}(\lambda, r_{0})e^{\frac{i}{4}\lambda^{2}\sigma_{3}}\lambda^{-i\nu\sigma_{3}},
\end{align*}
where
\begin{align*}
\mathcal{P}(\lambda, r_{0})=\left\{\begin{aligned}
&\left(
                    \begin{array}{cc}
                      1 & 0 \\
                      -r_{0} & 1 \\
                    \end{array}
                  \right),\quad &\lambda\in\Omega_{1},\\
&\left(
                    \begin{array}{cc}
                      1 & -\frac{r^{*}_{0}}{1+|r_{0}|^{2}} \\
                      0 & 1 \\
                    \end{array}
                  \right),\quad &\lambda\in\Omega_{3},\\
&\left(
                    \begin{array}{cc}
                      1 & 0\\
                      \frac{r_{0}}{1+|r_{0}|^{2}} & 1 \\
                    \end{array}
                  \right),\quad &\lambda\in\Omega_{4},\\
&\left(
                    \begin{array}{cc}
                      1 & r^{*}_{0} \\
                      0 & 1 \\
                    \end{array}
                  \right),\quad &\lambda\in\Omega_{6},\\
&~~~\emph{I},\quad &\lambda\in\Omega_{2}\cup\Omega_{5},
\end{aligned}\right.
\end{align*}
and
\begin{align*}
\Phi(\lambda, r_{0})=\left\{\begin{aligned}
\left(
                    \begin{array}{cc}
                      e^{-\frac{3\pi\nu}{4}}D_{i\nu}\left( e^{-\frac{3i\pi}{4}}\lambda\right) & -i\beta_{12}e^{-\frac{\pi}{4}(\nu-i)}D_{-i\nu-1}\left( e^{-\frac{i\pi}{4}}\lambda\right) \\
                      i\beta_{21}e^{-\frac{3\pi(\nu+i)}{4}}D_{i\nu-1}\left( e^{-\frac{3i\pi}{4}}\lambda\right) & e^{\frac{\pi\nu}{4}}D_{-i\nu}\left( e^{-\frac{i\pi}{4}}\lambda\right) \\
                    \end{array}
                  \right),\quad \lambda\in\mathbb{C}^{+},\\
\left(
                    \begin{array}{cc}
                      e^{\frac{\pi\nu}{4}}D_{i\nu}\left( e^{\frac{i\pi}{4}}\lambda\right) & -i\beta_{12}e^{-\frac{3\pi(\nu-i)}{4}}D_{-i\nu-1}\left( e^{\frac{3i\pi}{4}}\lambda\right) \\
                      i\beta_{21}e^{\frac{\pi}{4}(\nu+i)}D_{i\nu-1}\left( e^{\frac{i\pi}{4}}\lambda\right) & e^{-\frac{3\pi\nu}{4}}D_{-i\nu}\left( e^{\frac{3i\pi}{4}}\lambda\right) \\
                    \end{array}
                  \right),\quad \lambda\in\mathbb{C}^{-},
\end{aligned}\right.
\end{align*}
with
\begin{align*}
\beta_{12}=\frac{\sqrt{2\pi}e^{i\pi/4}e^{-\pi\nu/2}}{r_0\Gamma(-i\nu)},\quad \beta_{21}=\frac{-\sqrt{2\pi}e^{-i\pi/4}e^{-\pi\nu/2}}{r_0^*\Gamma(i\nu)}=\frac{\nu}{\beta_{12}}.
\end{align*}
Then, it is not hard to obtain the asymptotic behavior of the solution by using the well known asymptotic behavior of $D_{a}(z)$,
\begin{align}\label{A-1}
M^{(pc)}(r_0,\lambda)=I+\frac{M_1^{(pc)}}{i\lambda}+O(\lambda^{-2}), \tag{A.6}
\end{align}
where
\begin{align*}
M_1^{(pc)}=\begin{pmatrix}0&\beta_{12}\\-\beta_{21}&0\end{pmatrix}.
\end{align*}

\section*{Appendix B: Detailed calculations for the pure $\bar{\partial}$-Problem  }
\begin{prop}
For $t>0$, there exists constants $c_{j}(j=1,2,3)$ such that $I_{j}(j=1,2,3)$  which defined in \eqref{7.8} and \eqref{7.9} possess the following estimate
\begin{align}\label{B-1}
I_{j}\leq c_{j}t^{-\frac{1}{4}},~~ j=1,2,3. \tag{B.1}
\end{align}
\end{prop}
\begin{proof}
Let $s=p+iq$ and $z=\xi+i\eta$. Using the fact that
\begin{align*}
\Big|\Big|\frac{1}{s-z}\Big|\Big|_{L^{2}}(q+z_{0})=(\int_{q+z_{0}}^{\infty}\frac{1}{|s-z|^{2}}dp)^{\frac{1}{2}}
\leq\frac{\pi}{q-\eta},
\end{align*}
we can derive that
\begin{align}\label{B-2}
\begin{split}
|I_{1}|&\leq\int_{0}^{+\infty}\int_{q+z_{0}}^{+\infty}
\frac{|\bar{\partial}\chi_{\mathcal{Z}}(s)|e^{-8tq(p-z_{0})}}{|s-z|}dpdq\\
&\leq\int_{0}^{+\infty}e^{-8tq^{2}}\big|\big|\bar{\partial}\chi_{\mathcal{Z}}(s)\big|\big|_{L^{2}}(q+z_{0})
\Big|\Big|\frac{1}{s-z}\Big|\Big|_{L^{2}}(q+z_{0})dq \\
&\leq c_{1}\int_{0}^{+\infty}\frac{e^{-8tq^{2}}}{\sqrt{|q-\eta|}}dq
\leq c_{1}t^{-\frac{1}{4}}.
\end{split}\tag{B.2}
\end{align}
Similarly, considering that $r\in H^{1,1}(\mathbb{R})$, we obtain the estimate
\begin{align}\label{B-3}
|I_{2}|\leq\int_{0}^{+\infty}\int_{q+z_{0}}^{+\infty}
\frac{|r'(p)|e^{-8tq^{2}}}{|s-z|}dpdq
\leq c_{2}t^{-\frac{1}{4}}.\tag{B.3}
\end{align}
To obtain the estimate of $I_{3}$, we consider the following $L^{k}(k>2)$ norm
\begin{align}\label{B-4}
\Big|\Big|\frac{1}{\sqrt{|s-z_{0}|}}\Big|\Big|_{L^{k}}
\leq \left(\int_{q+z_{0}}^{+\infty}
\frac{1}{|p-z_{0}+iq|^{\frac{k}{2}}}dp\right)^{\frac{1}{k}}
\leq cq^{\frac{1}{k}-\frac{1}{2}}.\tag{B.4}
\end{align}
Similarly, we can derive that
\begin{align}\label{B-5}
\Big|\Big|\frac{1}{|s-z|}\Big|\Big|_{L^{k}}\leq c|q-\eta|^{\frac{1}{k}-\frac{1}{2}}.\tag{B.5}
\end{align}
By applying \eqref{B-4} and \eqref{B-4}, it is not hard to check that
\begin{align}\label{B-6}
\begin{split}
|I_{3}|&\leq\int_{0}^{+\infty}\int_{q}^{+\infty}
\frac{|z-z_{0}|^{-\frac{1}{2}}e^{-8tq(p-z_{0})}}{|s-z|}dpdq\\
&\leq\int_{0}^{+\infty}e^{-8tq^{2}}\Big|\Big|\frac{1}{\sqrt{|s-z_{0}|}}\Big|\Big|_{L^{k}}
\Big|\Big|\frac{1}{|s-z|}\Big|\Big|_{L^{k}}dq \leq c_{3}t^{-\frac{1}{4}}.
\end{split}\tag{B.6}
\end{align}
Now, we complete the estimates of $I_{j}(j=1,2,3)$.
\end{proof}

\renewcommand{\baselinestretch}{1.2}


\begin{thebibliography}{00}\addtolength{\itemsep}{-1.5ex}

\bibitem{NLS-1}
V.E. Zakharov, Stability of periodicwaves of finite amplitude on the surface of
a deep fluid, Sov. Phys. J. Appl. Mech. Tech. Phys. 4 (1968) 190-194.
\bibitem{NLS-2}
V.E. Zakharov, Collapse of Langmuir waves, Sov. Phys. JETP, 35 (1972) 908-914.
\bibitem{NLS-3}
 V.E. Zakharov , A.B. Shabat, Exact theory of two-dimensional self-focusing and one-dimensional self-modulation of waves in nonlinear media, Sov. Phys. JETP, 34 (1972) 62-69.
\bibitem{NLS-4}
 G.P. Agrawal, Nonlinear Fiber Optics Academic, San Diego, 1989.
\bibitem{NLS-5}
A. Hasegawa and Y. Kodama, Solitons in Optical Communications Clarendon, Oxford, 1995.


\bibitem{Tian-PAMS}
S.F. Tian, T.T. Zhang, Long-time asymptotic behavior for the Gerdjikov-Ivanov type of derivative nonlinear Schr\"{o}dinger equation with time-periodic boundary condition, Proc. Am. Math. Soc. 146 (2018) 1713-1729.
\bibitem{I-13}
S.F. Tian, Initial-boundary value problems for the general coupled nonlinear Schr\"{o}dinger equation on the interval via the Fokas method, J. Differ. Equ. 262 (2017) 506-558.
\bibitem{I-15}
S.F. Tian, The mixed coupled nonlinear Schr\"{o}dinger equation on the half-line via the Fokas method, Proc. R. Soc. Lond. A, 472(2195) (2016) 20160588.
\bibitem{Yanzy}
Z.Y. Yan, An initial-boundary value problem for the integrable spin-1 Gross-Pitaevskii equations with a $4\times 4$ Lax pair on the half-line, Chaos, 27(5) (2017) 053117.
\bibitem{Guobl}
B. Guo, N. Liu, Y. Wang, A Riemann-Hilbert approach for a new type coupled
nonlinear Schr\"{o}dinger equations, J. Math. Anal. Appl. 459(1) (2018) 145-158.
\bibitem{Wangds-2019-JDE}
D.S. Wang, B. Guo, X. Wang, Long-time asymptotics of the focusing Kundu-Eckhaus
equation with nonzero boundary conditions, J. Differ. Equ. 266(9) (2019) 5209-5253.

\bibitem{Manakov-system}
S.V. Manakov. On the theory of two-dimensional stationary selffocussing of electromagnetic waves. Sov. Phys. JETP, 65 (1973) 505-516.

\bibitem{K-E-geng}
X. Geng, H. W. Tam, Darboux Transformation and Soliton Solutions for Generalized Nonlinear Schr\"{o}dinger Equations. J. Phys. Soc. Jpn. 68(5) (1999), 1508-1512.
\bibitem{K-E-Tu}
G.Z. Tu, Liouville Integrability of Zero Curvature Equations, (Springer, Berlin) 1990.
\bibitem{K-E-JPSJ}
S. Kakei, N. Sasa, J. Satsuma, Bilinearization of a Generalized Derivative Nonlinear Schr\"{o}dinger Equation, J. Phys. Soc. Jpn. 64 (1995) 1519.
\bibitem{mll-1}
F. Zhao, Z. D. Li, Q. Y. Li, L. Wen, G. S. Fu, W. M. Liu, Magnetic rogue wave in a perpendicular anisotropic ferromagnetic nanowire with spin-transfer torque, Ann. Phys. 327(9) (2012) 2085-2095.
\bibitem{mll-peng}
W.Q. Peng, S.F. Tian, Long-time asymptotics in the modified Landau-Lifshitz equation with nonzero boundary conditions, arXiv:1912.00542.




\bibitem{Manakov-1974}
S.V. Manakov, Nonlinear Fraunhofer diffraction, Sov. Phys. JETP 38 (1974) 693-696.
\bibitem{Zakharov-1976}
V. E. Zakharov, S. V. Manakov, Asymptotic behavior of nonlinear wave systems integrated by the inverse scattering method, Sov. Phys. JETP, 44 (1976) 106-112.
\bibitem{Deift-1993}
P. Deift, X. Zhou, A steepest descent method for oscillatory Riemann¨CHilbert problems. Asymptotics for the MKdV equation, Ann. Math. 137(2) (1993) 295-368.
\bibitem{Deift-1994-1}
P. Deift, X. Zhou, Long-time asymptotics for integrable systems. Higher order theory, Comment. Phys.-Math.  165(1) (1994) 175-191.
\bibitem{Deift-1994-2}
P. Deift, X. Zhou, Long-Time Behavior of the Non-Focusing Nonlinear Schr\"{o}dinger Equation, a Case Study, Lectures in Mathematical Sciences, New Ser., vol. 5, Graduate School of Mathematical Sciences, University of Tokyo, 1994.
\bibitem{Deift-2003}
P. Deift, X. Zhou, Long-time asymptotics for solutions of the NLS equation with initial data in a weighted Sobolev space, Commun. Pure Appl. Math. 56(8) (2003) 1029-1077.
\bibitem{McLaughlin-1}
K. T. R. McLaughlin, P. D. Miller, The $\bar{\partial}$ steepest descent method and the asymptotic behavior of polynomials orthogonal on the unit circle with fixed and exponentially
varying non-analytic weights, Int. Math. Res. Not. (2006), Art. ID 48673.
\bibitem{McLaughlin-2}
K. T. R. McLaughlin, P. D. Miller, The $\bar{\partial}$ steepest descent method for orthogonal
polynomials on the real line with varying weights, Int. Math. Res. Not. IMRN (2008), Art. ID 075.

\bibitem{Dieng-2008}
 M. Dieng, K. D. T. McLaughlin, Long-time Asymptotics for the NLS equation via dbar
methods, arXiv: 0805.2807.
\bibitem{Cuccagna-2016}
S. Cuccagna, R. Jenkins, On asymptotic stability of $N$-solitons of the defocusing nonlinear Schr\"{o}dinger equation, Comm. Math. Phys. 343 (2016) 921-969.

\bibitem{AIHP}
M. Borghese, R. Jenkins, K. T. R. McLaughlin, Long-time asymptotic behavior of the
focusing nonlinear Schr\"{o}dinger equation, Ann. I. H. Poincar\'{e} Anal, 35 (2018) 887-920.
\bibitem{Jenkins}
R. Jenkins, J. Liu, P. Perry, C. Sulem, Soliton Resolution for the derivative nonlinear
Schr\"{o}dinger equation, Commun. Math. Phys. 363 (2018) 1003-1049.
\bibitem{Jenkins2}
R. Jenkins, J. Liu, P. Perry, C. Sulem, Global well-posedness for the derivative nonlinear
Schr\"{o}dinger equation, Commun. Part. Diff. Equ. 43(8) (2018) 1151-1195.
\bibitem{Faneg-1}
Y. L. Yang, E.G. Fan, Soliton Resolution for the Short-pluse Equation, arXiv:2005.12208.
\bibitem{Faneg-2}
Q. Y. Cheng, E.G. Fan, Soliton resolution for the focusing Fokas-Lenells equation with weighted Sobolev initial data, arXiv:2010.08714.
\bibitem{Faneg-3}
R. H. Ma, E.G. Fan, Longtime asymptotic behavior of the focusing nonlinear Kundu-Eckhaus equation, arXiv:1912.01425.

\bibitem{Beals-IP}
R. Beals, R.R. Coifman, Linear spectral problems, non-linear equations and the $\bar{\partial}$-method, Inverse Probl. 5 (1989) 87-130.

\bibitem{Li-1}
Z.Q. Li, S.F. Tian, J.J. Yang, Riemann-Hilbert approach and soliton solutions for the higher-order dispersive nonlinear Schr\"{o}dinger equation with nonzero boundary conditions, arXiv:1911.01624.

\bibitem{PC-model}
A. Its, Asymptotic behavior of the solutions to the nonlinear Schr\"{o}dinger equation, and isomonodromic deformations of systems of linear
differential equations, Dokl. Akad. Nauk SSSR (ISSN 0002-3264) 261(1) (1981) 14-18.
\bibitem{PC-model-2}
J. Liu, P. Perry, C. Sulem, Long-time behavior of solutions to the derivative nonlinear
Schr\"{o}dinger equation for soliton-free initial data, Ann. I. H. Poincar\'{e}, Anal. Non Lin\'{e}aire, 35 (2018) 217-265.
\bibitem{PC-equation}
F.W.J. Olver, A.B. Olde Daalhuis, D.W. Lozier, B.I. Schneider, R.F. Boisvert, C.W. Clark, B.R. Miller, B.V. Saunders, NIST Digital Library of Mathematical Functions, (2016). http://dlmf.nist.gov/.
\bibitem{PC-solution2}
R. Jenkins, K. McLaughlin, Semiclassical limit of focusing NLS for a family of square barrier initial data, Commun. Pure Appl. Math. 67 (2) (2014) 246-320.

\end{thebibliography}
\end{document}